\documentclass[12pt,a4paper,reqno]{amsart}
\usepackage{amsfonts}
\usepackage{lscape}
\usepackage{amsmath}
\usepackage{amssymb}
\usepackage{geometry}
\usepackage{latexsym}
\usepackage{afterpage}
\usepackage{amstext}
\usepackage{amscd}
\usepackage{setspace}
\usepackage{caption}
\usepackage{graphicx}
\usepackage{enumerate}
\usepackage{bbm}
\usepackage{hyperref}

\usepackage{color}
\usepackage{xcolor}

\newtheorem{theorem}{Theorem}[section]
\numberwithin{equation}{section}

\newtheorem{corollary}[theorem]{Corollary}

\newtheorem{definition}[theorem]{Definition}
\newtheorem{example}[theorem]{Example}

\newtheorem{lemma}[theorem]{Lemma}

\newtheorem{proposition}[theorem]{Proposition}
\newtheorem{remark}[theorem]{Remark}

\newtheorem*{Theorem*}{Theorem}
\newtheorem{maintheorem}{Theorem}

\newtheorem{maincorollary}[maintheorem]{Corollary}

\newcommand{\cmc}{\begin{maincorollary}}
\newcommand{\fmc}{\end{maincorollary}}

\renewcommand{\varepsilon}{\epsilon}

\newcommand{\m}{m}

\DeclareMathOperator{\Cor}{Cor}

\title[Statistical properties of dynamical systems via induced WGM maps]
{Statistical properties of dynamical systems\\ via induced weak Gibbs Markov maps}

\author[A. Ullah]{Asad Ullah}
\address{Centro de Matem\'atica e Aplica\c{c}\~oes (CMA-UBI), Universidade da Beira
	Interior, Rua Marqu\^es d'\'Avila e Bolama, 6201-001, Covilh\~a, Portugal.}
\email{asad.ullah@ubi.pt}

\author[H. Vilarinho]{Helder Vilarinho}
\email{helder@ubi.pt}

\date{\today}



\thanks{A. Ullah and H. vilarinho were partially supported by Funda\c{c}\~ao para a Ci\^encia e a Tecnologia (FCT) 
through Centro de Matemática e Aplicações (CMA-UBI), Universidade da Beira Interior, under the project
UIDB/00212/2020. A. Ullah was also supported by FCT under the grant number UI/BD/150796/2020.}

\begin{document}

\begin{abstract}
In this article, we address the decay of correlations for dynamical systems that admit an induced weak Gibbs Markov map (not necessarily full branch). Our approach generalizes L.-S. Young's coupling arguments to estimate the decay of correlations for the tower map of the induced weak Gibbs Markov map in terms of the tail of the return time function. For that we initially discuss how to ensure the mixing property of the tower map. Additionally, we yield results concerning the Central Limit Theorem and Large Deviations.
\end{abstract}

 \maketitle 



\section{Introduction}
 The induced scheme developed by L.-S. Young in  \cite{Y98, Y99} is certainly among the most powerful tools for studying the statistical properties of non-uniformly hyperbolic dynamical systems. Roughly speaking,  it involves considering a region of the phase space partitioned into subdomains where a  return time is defined (not necessarily the first one) and establishing an induced map within that region that satisfies several good properties.
In this approach there is an explicit connection between the tail of the return times and the decay of correlations for the map, at least for some specific rates.
In \cite{Y98}, the existence of  a hyperbolic structure for given dynamical systems was assumed, from which mixing rates at exponential speeds were obtained.
On the other hand, in \cite{Y99} an abstract model is considered allowing different decay rates. More precisely, 
it was introduced a suitable dynamical system, nowadays known as Young towers, based on an induced full branch  map, referred as Gibbs Markov (GM) map, and some statistical properties for this tower system were derived. The idea is then to carry those results to a dynamical system with such induced scheme.
A Central Limit Theorem (CLT) was also obtained in both works.
   
The aim of this paper is to obtain similar to those in~\cite{Y99} under weaker assumptions where the induced map is not necessarily full branch, that we refer as a \textit{weak Gibbs Markov} (WGM) map (see Definition~\ref{def:WGM}).

Constructing an induced scheme with the desired properties is often a difficult task. Furthermore, if an induced WGM map is obtained, transitioning to a full branch scenario is not always straightforward.
There are some examples in the literature where an induced WGM map is constructed. In some cases, as for instance in \cite[\S6]{Y99}, \cite{Gouezel reproducing full return GM} and \cite{Peyman Paper}, the induced WGM map is used to get a GM map in order to apply the abstract results from \cite{Y99}.  In other situations, the step to get a full branch map seems   extremely laborious or may not be possible; c.f. \cite{Gouezel skew product}. 

The list of works concerning the construction of an induced GM map and the derived statistical properties is too extensive to be included here. Some references focusing on the construction of induced GM map in one dimensional maps are e.g.~\cite{Decay of correlation via full return induced GM map, Wobbly Intermittent map, doubly intermittent map, Statistical properties via full return induced GM map, M.Holland construct induced map}, and for two dimensional maps refer to e.g.~\cite{Peyman Paper, Peyman and Melbourn}. Other works explore the consequences of having induced GM maps, as seen in~\cite{Invariance Principle via full return induced GM,Large deviation via full return induced GM}. We also refer to~\cite{Alves Book, About existnace of full return Induced GM map, extended result of L.S.Y, D.Corrlation for Henon map, H. Bruin upper and lower via inducing, Billiards with poly-mixing, Large deviation on Young tower,  Lynch paper, Decay of Correlation via renewal operator,  Sarig Polynomial decay of correlation via renewal operator, exponentional decay of correlation via coupling cones} and references therein for just an overview of the theory.
In the opposite direction, it was shown in \cite{expanding measure implies full return GM, statistical properties implies full return GM} that for a large class of maps, having an induced GM map is a necessary condition for some statistical properties. 

This work provides abstract results to be applied to dynamical systems with an induced WGM map, even when the existence of an induced GM map is not guaranteed. There are some previous works in this context. 
In~\cite{Gouezel Polynomial decay of correlation via renewal operator} the author considers a topologically mixing Markov map, preserving the reference probability measure and an induced WGM map was defined with respect to this reference measure. An operator renewal theory  was used as a method to obtain sharp polynomial bounds for decay of correlations and also a CLT, both for observables supported in the inducing region. In contrast, we do not assume that the reference measure is preserved by the dynamics and that the observables are supported in the induced region. As a method we follow L.-S. Young's coupling arguments to get estimates on decay of correlations, CLT and Large Deviations.
In~\cite{V. Maume Decay of correlation on Tower} the author addresses the decay of correlations for a mixing tower map associated with an induced WGM map via Birkhoff's cones and projective metrics. This work was extended in~\cite{J. Buzzi Paper} allowing the mixing tower map to have a non-H\"older Jacobian. In particular the Gibbs property (bounded distortion) for the tower map is replaced by a summable variation condition. The setting considered in present work is different than \cite{J. Buzzi Paper, V. Maume Decay of correlation on Tower} as we provide sufficient conditions on the induced map to obtain a mixing tower system. Moreover, due to the different techniques, the estimates on the decay of correlation we obtain are slightly better.

In this work, we assume the existence of an induced WGM map. With additional hypothesis on the induced map we build a mixing tower map. We obtain the decay of correlations for this tower map  in terms of the tail of the return time function. 
The results then go back to the original system through a semi-conjugacy, providing an abstract result for dynamical systems admitting induced WGM maps. We also deduce a CLT and Large Deviations for those dynamical systems.

We would like to share some remarks concerning the lack of full branch property:
\begin{enumerate}[i)]
	\item The density of the invariant measure for the WGM map does not necessarily have a positive lower bound (Theorem~\ref{ergodic measure for irreducble WGM}). This positive lower bound is commonly known to exist for GM maps. 	
	\item The invariant ergodic probability measure for the tower map of an induced WGM map $f^R$ is not mixing, even under the condition $\gcd\{R\}=1$. We provide sufficient conditions to ensure the invariant measure to be mixing (see Section~\ref{Mixing measure for TM}).
	\item The choice of $ n_{0} $ in the definition of the sequence of \emph{stopping times} (Section~\ref{se:dc on tower}) is made by using the long branch property in such a way that we obtain ~\cite[Lemmas 1 and 2]{Y99} without the full branch property.
	\item The \emph{weak} hypothesis does not guarantee a symmetry as in~\cite[(1)]{Y99} (see also~\cite[Lemmas 3.39 and 3.43]{Alves Book}). This implies that the quantity $ I_{2} $ in Proposition~\ref{first Matching Proposition} does not vanish. 
	\end{enumerate}

This paper is organized as follows.
In Section \ref{relvant def and main results},  we introduce the necessary definitions and outline the main results.  In Section \ref{Properties of WGM}, we recall some theory about WGM maps. In Section \ref{Mixing measure for TM}, we discuss the mixing invariant probability measure for the tower map.
In Section \ref{se:dc on tower}, we obtain the decay of correlation for the tower map and then we conclude the results for the original dynamical system. Lastly, in Section~\ref{sec:CLT}, we discuss CLT.

\section{Preliminaries and statement of main results}\label{relvant def and main results}
Consider a measure space $(\Delta_0,\mathcal A,\m)$, with $0<m(\Delta_{0})<\infty$, a measurable map $ F: \Delta_{0}\rightarrow \Delta_{0}$ and an ($\m$ mod 0) finite or countable partition $\mathcal{P}_{0}$ of $ \Delta_{0} $ into invertibility  domains of $F$, that is, $F$ is a bijection from each $\omega\in \mathcal{P}_{0}$ to $F(\omega)$, with measurable inverse.
For each $n\geq0$ we set
$$ F^{-n}(\mathcal{P}_{0})=\{F^{-n}(\omega): \omega\in \mathcal{P}_{0}\}.$$
Assuming $F_{*}m\ll m$, we have that $ F^{-i}(\mathcal{P}_{0}) $ is an ($\m$ mod 0) partition of $ \Delta_{0} $, for all $i\geq 0$.
In this case we have a sequence $(\mathcal P_0^n)_n$ of ($\m$ mod 0) partitions given by
\begin{equation*}
	\mathcal{P}_{0}^{n}=\bigvee_{i=0}^{n-1}F^{-i}(\mathcal{P}_{0})=\{\omega_{0}\cap F^{-1}(\omega_{1})\cap...\cap F^{-n+1}(\omega_{n-1}): \omega_{0},\ldots,\omega_{n-1}\in \mathcal{P}_{0}\},
\end{equation*} 
for $n\geq 1$, and the ($\m$ mod 0) partition
\begin{equation*}
\mathcal P_0^\infty=\bigvee_{i=0}^{\infty}F^{-i}(\mathcal{P}_{0})=\{\omega_{0}\cap F^{-1}(\omega_{1})\cap\cdots: \omega_{n}\in \mathcal{P}_{0} \text{ for all } n\geq0\}.
\end{equation*}
We say that the sequence $(\mathcal P_0^n)_{n}$ is \textit{a basis of  $\Delta_{0}$} if it generates $\mathcal{A}$ ($\m$ mod 0) and $\mathcal P_0^\infty$ is the partition into single points.

\begin{definition}\label{def:WGM}
We say that $F: \Delta_{0}\rightarrow \Delta_{0}$ is a weak Gibbs Markov (WGM) map, with respect to the  partition $\mathcal{P}_{0}$, if the following hold:
\begin{enumerate}[W$1$)]
\item \label{def:WGM:Markov} Markov: $F$ maps each $\omega\in \mathcal{P}_{0}$ bijectively to an ($\m$ mod 0) union of elements of $\mathcal{P}_{0}$.

\item Separability: the sequence $(\mathcal P_0^n)_{n}$ is a basis of  $\Delta_{0}$.

\item Nonsingular: there exists a strictly positive measurable function $J_{F}$ defined on $\Delta_0$ such that for each $A\subset\omega\in \mathcal{P}_{0}$,
\begin{equation*}
	m(F(A))=\int_{A}J_{F}dm.
\end{equation*}
\item Gibbs: there exist $C_{F}>0$ and $0<\beta<1$ such that for all  $\omega\in \mathcal{P}_{0}$ and $x,y\in \omega$,
\begin{equation*}
	\log\frac{J_{F}(x)}{J_{F}(y)}\leq C_{F}\beta^{s(F(x), F(y))},
\end{equation*}
where 
$$s(x,y)=min\{n\geq0:F^{n}(x) \mbox{ and } F^{n}(y) \mbox{ lie in distinct elements of } \mathcal{P}_{0} \}.$$
\item \label{def:WGM:Long} Long branches: there exists $ \delta_{0}>0 $ such that $ m(F(\omega))\geq \delta_{0} $, for all $ \omega\in \mathcal{P}_{0}$.
\end{enumerate}
\medskip
\end{definition}
The term \emph{weak} in the above definition refers that in W\ref{def:WGM:Markov}) we do not require \textit{full branch} (that is, $F(\omega)=\Delta_0$ for all $\omega\in\mathcal P_0$). In the case of full branch, $F$ is called a \textit{Gibbs Markov} map. This terminology was introduced in \cite{Alves Book}.
\medskip

Consider a measure space $(M,\mathcal B,\m)$, a measurable map $f: M\rightarrow M$ and $\Delta_{0}\subset M$ with $m(\Delta_{0})>0$. For simplicity we denote the restriction of $\m$ to $\Delta_0$ also by $\m$. We say that $ F:\Delta_{0}\rightarrow \Delta_{0} $ is an \emph{induced map for $f$} if there exists
a countable ($\m$ mod 0) partition $\mathcal{P}_{0}$ of $\Delta_{0}$ and a measurable function $R:\Delta_{0}\rightarrow \mathbb{N}$, constant on each elements of $\mathcal{P}_{0}$, such that 
$$F|_{\omega}=f^{R(\omega)}|_{\omega}$$ 
We formally denote the induced map $F$ by $f^{R}$.
We say that an induced map $f^R\colon\Delta_{0}\rightarrow\Delta_{0}$ is  \emph{aperiodic} if for all  $\omega_{1}, \omega_{2}\in\mathcal{P}_{0}$ there exists  $k_{0}\in \mathbb{N}$ such that  for all  $n\geq k_{0}$,
\begin{equation*}
m(\omega_{1}\cap (f^R)^{-n}(\omega_{2}))>0.
\end{equation*}
We say that an induced map $f^R$ has a \emph{coprime block} if there exist $N\geq2$ and $ \omega_{1},\omega_{2},...,\omega_{N}\in\mathcal{P}_0 $
 such that $\gcd\{R(\omega_i)\}_i=1$ and for all $i=1,\ldots,N$,
 \begin{equation*}
 f^{R}(\omega_{i})\supseteq\omega_{1}\cup\omega_{2}\cup\cdots\cup\omega_{N}\,\,\,(\m\text{ mod } 0).
 \end{equation*}

Assume  that $M$ is a metric space with metric $d$. 
 We say that an induced map $f^{R}$ is \emph{expanding} if there are $C>0$ and $0 < \beta < 1$ such that, for all $\omega \in \mathcal{P}_0$ and $x,y \in \omega$
\begin{itemize}
	\item[i)] $d(f^{R}(x), f^{R}(y)) \leq C \beta^{s(x,y)}$,
	\item[ii)] $d(f^{j}(x), f^{j}(y)) \leq Cd(f^{R}(x), f^{R}(y)),$ for all $0 \leq j \leq R.$
\end{itemize}
For given
$\eta>0$, a function $\varphi\colon M \rightarrow \mathbb{R}$ is \emph{$\eta$-H\"{o}lder continuous} if
 \begin{equation*}
 	|\varphi|_{\eta} = \sup_{x \neq y} \frac{|\varphi(x) - \varphi(y)|}{d(x,y)^{\eta}} < \infty.
 \end{equation*}
 We set
 \begin{equation*}
 	\mathcal{H}_{\eta} = \{ \varphi \colon M \rightarrow \mathbb{R} \colon \varphi \text{ is }  \eta\text{-H\"{o}lder continuous}\}.
 \end{equation*}

The correlation sequence of two observable functions $\varphi \in \mathcal{H}$, for some Banach space $\mathcal{H}\subseteq L^\infty(M,\m)$, and $\psi \in L^\infty(M,\m)$, with respect to an $f$-invariant probability $\mu$, is defined for all $n\in\mathbb N$ by
 \begin{equation*}
 	\Cor_{\mu}(\varphi, \psi \circ f^{n})=\left|\int\varphi(\psi\circ f^{n}) d\mu-\int\varphi d\mu\int\psi d\mu\right|.
 \end{equation*}

\begin{maintheorem}\label{teo:main}
Consider a measure space $(M,\mathcal B,\m)$, endowed with some metric, and a measurable map $f: M\rightarrow M$ satisfying $f_{*}m \ll m$. Let $f^{R}: \Delta_{0} \rightarrow \Delta_{0}$ be an aperiodic induced WGM expanding map with a coprime block and $R \in L^{1}(\m)$. Then there exists a unique ergodic $f$-invariant probability measure  $\mu$ such that $\mu \ll m$ and $\mu(\Delta_{0})>0$. Moreover,
\begin{enumerate}
\item if $m\{ R >n \} \leq C n^{-a}$ for some $C>0$ and $a >1$, then for all $\varphi \in \mathcal{H}_{\eta}$ and $\psi \in L^\infty(M,\m),$ there is some $C^{'} > 0$ such that
$$
\Cor_{\mu}(\varphi, \psi \circ f^{n}) \leq C^{'} n^{-a +1};
$$
\item if $m\{ R >n \} \leq C e^{-cn^{a}}$ for some $C,c>0$ and $0<a\leq 1$, given $\eta>0$ there is $c'>0$ such that, for all $\varphi \in \mathcal{H}_{\eta}$ and $\psi \in L^\infty(M,\m),$ there is some $C{'} > 0$ such that
$$
\Cor_{\mu}(\varphi, \psi \circ f^{n}) \leq C^{'} e^{-c{'}n^{a}}.
$$
\end{enumerate}
\end{maintheorem}

The proof of Theorem~\ref{teo:main} is given in Section~\ref{se:dc on tower}. The strategy consists into reduce the estimates on the decay of correlations to similar ones for a tower map.
We remark that the coprime block is used to get $ \mu $ to be mixing. We will explore the significance of the coprime block in Section \ref{Mixing measure for TM}.
\medskip

Let $ \mu $ be an ergodic $ f$-invariant probability measure. We say that an observable $ \varphi: M \rightarrow \mathbb{R} $ with $ \int\varphi d\mu=0 $ satisfies the \textit{Central Limit Theorem (CLT)} if $ \frac{1}{\sqrt{n}}\sum_{i=0}^{n-1} \varphi \circ f^{i} $ converges in law (or in distribution) to a normal distribution $ \mathcal{N}(0,\sigma)$, for some $ \sigma>0.$ We may also consider observables of non-zero expectation by replacing $ \varphi $ with $ \varphi-\int \varphi d\mu$. In this situation an observable $ \varphi $ satisfies the CLT if there exists $ \sigma>0 $ such that for every interval $ J\subset \mathbb{R} $,
\begin{equation*}
	\mu \left\{ x : \frac{1}{\sqrt{n}}\sum_{i=0}^{n-1}\left(\varphi( f^{i}(x))- \int \varphi d\mu \right)\in J \right\} \rightarrow \frac{1}{\sigma \sqrt{2\pi}}\int_{J}e^{-\frac{t^2}{2\sigma^2}}dt, \text{ as } n\rightarrow \infty.
\end{equation*}
A function $ \varphi $ is \textit{coboundary} if there exists a measurable function $ g $   such that $ \varphi\circ f=g\circ f-g$.  
\begin{maincorollary}[Central Limit Theorem]\label{CLT Main Cror}
	Consider a measure space $(M,\mathcal B,\m)$, endowed with some metric, and a measurable map $f: M\rightarrow M$ satisfying $f_{*}m \ll m$. Let $f^{R}: \Delta_{0} \rightarrow \Delta_{0}$ be an aperiodic induced WGM expanding map with a coprime block such that $R \in L^{1}(\m)$, and let $\mu$ be the unique ergodic $f$-invariant probability measure  $\mu$ such that $\mu \ll m$ and $\mu(\Delta_{0})>0$. 
		If $m\{ R >n \} \leq C n^{-a}$ for some $C>0$ and $a >2$, then the CLT is satisfied for all $\varphi \in \mathcal{H}_{\eta}$ if and only if $\varphi$ is not coboundary.
\end{maincorollary}

Given  $ \epsilon>0 $ we define the \textit{large deviation at time $ n $} of the time average of an observable $ \varphi\colon M\to\mathbb R $ from the spatial average as
\begin{equation*}
LD_{\mu}(\varphi, \epsilon, n):	=\mu\left(\left|\frac{1}{n}\sum_{i=0}^{n-1} \varphi\circ f^i - \int \varphi d\mu\right|> \epsilon \right).
\end{equation*} 
Birkhoff's ergodic theorem guarantees that time averages converges to the spatial average $\mu$-a.e., and so $ LD_{\mu}(\varphi, \epsilon, n) \rightarrow 0 $, as $ n\rightarrow \infty.$  We are interested in the rate of this decay. Combining Theorem~\ref{teo:main} and \cite{LD with improved polynomial rate, statistical properties implies full return GM}, we have:
\begin{maincorollary}[Large Deviations]\label{corol:LD}
	Consider a measure space $(M,\mathcal B,\m)$, endowed with some metric, and a measurable map $f: M\rightarrow M$ satisfying $f_{*}m \ll m$. Let $f^{R}: \Delta_{0} \rightarrow \Delta_{0}$ be an aperiodic induced WGM expanding map with a coprime block such that $R \in L^{1}(\m)$, and let $\mu$ be the unique ergodic $f$-invariant probability measure  $\mu$ such that $\mu \ll m$ and $\mu(\Delta_{0})>0$. Then,
	\begin{enumerate}
		\item If $m\{ R >n \} \leq C n^{-a}$ for some $C>0$ and $a >1$, then for all $ \epsilon>0 $ and  $\varphi \in \mathcal{H}_{\eta}$ there is some $C^{'}=C^{'}(\epsilon, \varphi) > 0$ such that
		$$
		LD_{\mu}(\varphi, \epsilon, n) \leq C^{'} n^{-a +1}.
		$$
		\item If $m\{ R >n \} \leq C e^{-cn^{a}}$ for some $C,c>0$ and $0<a\leq 1$, then for all $ \epsilon>0 $ and $\varphi \in \mathcal{H}_{\eta}$ there is some $C{'}=C^{'}(\epsilon, \varphi)> 0$ and $c'=c'(c, \varphi, \epsilon, \eta) $ such that
		$$
		LD_{\mu}(\varphi, \epsilon, n) \leq C^{'} e^{-c{'}n^{\frac{a}{a+2}}}.
		$$
	\end{enumerate}
\end{maincorollary}

\begin{example}
Let $M=S^1\times[0,1]$ and let $\m$ denote the Lebesgue measure on $M$.  We consider the map $f:M\rightarrow M$ introduced in~\cite{Gouezel skew product} defined by
$$
f(\theta, x)=\left(F(\theta), f_{\alpha(\theta)}(x)\right),
$$
where $F(\theta)=4\theta$,
$$
f_\alpha(\theta)(x)= \begin{cases}x\left(1+2^{\alpha(\theta)} x^{\alpha(\theta)}\right) & \text { if } 0 \leqslant x \leqslant \frac{1}{2} \\ 2 x-1 & \text { if } \frac{1}{2}<x \leqslant 1,\end{cases}
$$
and  $\alpha: S^1 \rightarrow(0, 1)$ is a $C^1$ map that has minimum $\alpha_{\min }$ and maximum $\alpha_{\max }$, with $\alpha_{\min }<\alpha_{\max }$. In \cite{Gouezel skew product} a partition $\mathcal P_0$ of $\Delta_0=S^1\times(\frac12,1]$ with a certain return time $R$ was given in such a way that $f^{R}$ is an aperiodic induced WGM map (no full branch property) with a coprime block. Moreover, we have
$	m\{R>n\}\leq {C}{n^{-a}}$
for some $C>0$, with $a=1/\alpha_{\max}$. Since $\alpha_{\max} <1 $ we have   $ R\in L^{1}(\m).$
By Theorem~\ref{teo:main} there exists a unique ergodic $f$-invariant probability measure  $\mu$ such that $\mu \ll m$ and for all $\varphi \in \mathcal{H}_{\eta}$ and $\psi \in L^\infty(M,\m)$, there is some $C^{'} > 0$ such that
$$
\Cor_{\mu}(\varphi, \psi \circ f^{n}) \leq C^{'} n^{-a +1}.
$$
By Corollary~\ref{CLT Main Cror}, if $\alpha_{max}<1/2$ (thus $a>2$) then CLT is satisfied for all $\varphi \in \mathcal{H}_{\eta}$ that is not coboundary, and by Corollary~\ref{corol:LD} we also have
$$		LD_{\mu}(\varphi, \epsilon, n) \leq C^{''} n^{-a +1}
		$$
for all $\varphi \in \mathcal{H}_{\eta}$ and some $C^{''}=C^{''}(\epsilon, \varphi) > 0$.

We notice that under additional hypothesis on $\alpha$ we can have better estimates on the return times (see \cite{Gouezel skew product}), hence in the decay of correlations, CLT and Large Deviations. 
It is worth to mention that a CLT for this map was already discussed in \cite{Gouezel skew product}. Moreover, a polynomial decay of correlations for observables supported on $\Delta_0$ was pointed out in~\cite{Gouezel thesis}.
\end{example}

\section{Invariant measures for WGM maps}\label{Properties of WGM}
In this section, for the convenience of the reader, we collect some results on WGM maps. For the missing proofs see \textit{e.g.} \cite{Alves Book}.

Given $0< \beta <1$, we define the following \emph{spaces of densities}:
\begin{equation*}
\mathcal{F}_{\beta}( \Delta _{0} ) = \{\varphi\colon\Delta_{0}\rightarrow \mathbb{R} \colon \exists C_\varphi > 0\, \text{ s.t. }
\, |\varphi(x)-\varphi(y)|\leq C_{\varphi}\beta^{s(x,y)}, \forall x,y \in \Delta_{0}\}
\end{equation*}
\begin{multline*}
	\mathcal{F}_{\beta}^{+}( \Delta _{0} ) = \{\varphi\in \mathcal{F}_{\beta}( \Delta _{0} ) \colon \exists C_{\varphi}' > 0 \, \text{ s.t. }\,\\ \varphi(x)>0 \text{ and } \left|\frac{\varphi(x)}{\varphi(y)}-1\right|\leq C_{\varphi}'\beta^{s(x,y)}, \forall x,y \in \omega\in \mathcal{P}_0 \}.
\end{multline*}
Given $\varphi \in \mathcal{F}_{\beta}(\Delta_{0})$, set 
\begin{equation*}
|\varphi|_{\beta}= \sup_{x \neq y}\frac{|\varphi(x)-\varphi(y)|}{\beta^{s(x,y)}}.
\end{equation*} 
Since $\mathcal{F}_{\beta}(\Delta_{0})\subset L^{\infty}(m) \subset L^{1}(m)$, we define a norm on $\mathcal{F}_{\beta}(\Delta_{0})$ as
\begin{equation*}
\|\varphi\|_{\beta} = |\varphi|_\beta + \|\varphi\|_\infty.
\end{equation*}
Set, for each $n\in \mathbb{N}$ and $\omega \in \mathcal{P}_{0}^n$,
$$ \rho_{n} = \frac{dF^{n}_{*}m}{dm}\,\,  \text{ and }\,\, \rho_{n,\omega} = \frac{dF^{n}_{*}(m|\omega)}{dm}.$$

\begin{lemma}\label{lem:WGMprop}
Let $F: \Delta_{0} \rightarrow \Delta_{0}$ be a WGM map. Then,
\begin{enumerate}[i)]
\item  there is some $\bar{C_{F}} >0$ such that, for all $n \geq 1$ and $x,y \in \omega \in \mathcal{P}_{0}^{n},$
\begin{equation*}
		\log\frac{J_{F^{n}}(x)}{J_{F^{n}}(y)} \leq \bar{C_{F}} \beta^{s(F^{n}(x), F^{n}(y))}. 
\end{equation*}
\item there is some $C> 0$ such that, for all $n \geq 1$, $\omega \in \mathcal{P}_{0}^{n}$ and measurable sets $A_{1}, A_{2} \subset \omega,$
\begin{equation*}
\frac{m(F^{n}(A_{1}))}{m(F^{n}(A_{2}))} \leq C \frac{m(A_{1})}{m(A_{2})}.
\end{equation*}\label{lem:WGMpropii}
\end{enumerate} 
 \end{lemma}

\begin{proposition}\label{RP1}
If $F: \Delta_{0}\rightarrow \Delta_{0}$ is a WGM map, then $F$ has an invariant probability measure $\nu_{0} \ll m$ with $d\nu_{0}/dm \in \mathcal{F}_{\beta}(\Delta_{0})$.
\end{proposition}

\begin{remark}\label{RD1}
The density ${d\nu_{0}}/{dm}$ given by Proposition \ref{RP1} can be obtained as the pointwise limit of a subsequence of $\rho_{n}$ (see for instance ~\cite[Theorem 3.13]{Alves Book} for more details). From the first item of Lemma~\ref{lem:WGMprop}, for each $ \omega \in \mathcal{P}_0 $, and $ x, y \in F(\omega) $  we have
\begin{equation*}
	\rho_{n}(x)\leq e^{\bar{C_{F}}\beta^{s(x,y)}}\rho_{n}(y) \text{ for every } n \geq 1,
\end{equation*}
which implies
\begin{equation}\label{RDE1}
\text{ either }\frac{d\nu_{0}}{dm}(x)=0\,\, \text{ or }\,\, \frac{d\nu_{0}}{dm}(x)>0 \text{ and } \left|\log\frac{\frac{d\nu_{0}}{dm}(x)}{\frac{d\nu_{0}}{dm}(y)}\right |\leq \bar{C_{F}}\beta^{s(x, y)}.
\end{equation}  
Note that  $\left|\log\frac{\frac{d\nu_{0}}{dm}(x)}{\frac{d\nu_{0}}{dm}(y)}\right |\leq \bar{C_{F}}\beta^{s(x, y)}$ implies that there exists $ C>0 $ such that
\begin{equation*}
	\left|\frac{\frac{d\nu_{0}}{dm}(x)}{\frac{d\nu_{0}}{dm}(y)}-1\right |\leq  C\beta^{s(x, y)}.
\end{equation*}
\end{remark}

We say that a WGM map $F$ is \emph{irreducible}, with respect to the given partition $\mathcal{P}_{0}$, if for all $\omega_{1}, \omega_{2} \in \mathcal{P}_{0}$, there exists $n\in \mathbb{N}$ such that $ m(\omega_{1}\cap F^{-n}(\omega_{2}))>0$.

\begin{theorem}\label{ergodic measure for irreducble WGM}
If $F: \Delta_{0}\rightarrow \Delta_{0}$ is an irreducible WGM map, then the $F$-invariant probability measure $ \nu_{0}\ll m$ is unique, ergodic, $ \frac{d\nu_{0}}{dm} \in \mathcal{F}_{\beta}^{+}(\Delta _{0})$ and there is $ C_{0} $ such that  $ 0<\frac{d\nu_{0}}{dm}\leq C_{0}.$
Moreover, if $F$ is aperiodic, then $\nu_{0}$ is exact.
\end{theorem}

\begin{proof}
The ergodicity and exactness of the measure $ \nu_0 $ is proved in \cite{Aar Paper}. Let us prove $ \frac{d\nu_{0}}{dm} \in \mathcal{F}_{\beta}^{+}( \Delta _{0} )$. From Proposition~\ref{RP1}, $d\nu_{0}/dm \in \mathcal{F}_{\beta}(\Delta_{0})$. 
Let $ \omega \in \mathcal{P}_0 $ be an arbitrary element. 
Since $F$ is an irreducible map, then $ \omega $ must intersect $ F(\omega_{0}) $ for some $\omega_{0}\in \mathcal{P}_0$, $\omega_0\neq \omega$, and this implies $ \omega \subset F(\omega_{0}) $. 
By Remark \ref{RD1}, ~\eqref{RDE1} holds for each $ x,y\in F(\omega_{0})$, and in particular it holds  for all $x, y\in \omega$.
So, its enough to prove $\frac{d\nu_{0}}{dm}>0$ on $\omega$.
Assume $\frac{d\nu_{0}}{dm}=0$ on $\omega$. By irreducibility, for any $ \omega^{\prime}\in \mathcal{P}_0 $  there exists $ n_0 \in \mathbb{N} $ such that $ m(\omega^{\prime}\cap F^{-n_0}(\omega)) >0$. By using the invariance of $ \nu_{0} $ we have 
\begin{equation*}
0=\nu_{0}(\omega)=\nu_{0}(F^{-n_0}(\omega))\geq \nu_{0}(\omega^{\prime}\cap F^{-n_0}(\omega)), 
\end{equation*}
and this means that we have $ \nu_{0}(\omega^{\prime}\cap F^{-n_0}\omega)=0 $. Then, we must have  $\frac{d\nu_{0}}{dm}=0$ on $\omega^{\prime}\cap F^{-n_0}(\omega)\subset \omega^{\prime}$.
From \eqref{RDE1}, we have $\frac{d\nu_{0}}{dm}=0$ on $ \omega^{\prime} $. Hence $ \nu_{0}(\omega^{\prime})=0 $ for any $ \omega^{\prime}\in \mathcal{P}_0 $, and this implies $ \nu_{0}(\Delta_{0}) =0$, which is a contradiction.
We conclude that $ \frac{d\nu_{0}}{dm}>0  $ on each  $ \omega \in \mathcal{P}_0$. Since $d\nu_{0}/dm \in \mathcal{F}_{\beta}(\Delta_{0})$  there is $ C_{0}>0$  such that $ 0<\frac{d\nu_{0}}{dm}\leq C_{0}$, which also implies the uniqueness of $\nu_{0}$. 
\end{proof}

\begin{remark}\label{gives ergodic measure for f}
If $f^{R}: \Delta_{0}\rightarrow \Delta_{0}$ is an irreducible induced WGM map for $f$ then, by Theorem~\ref{ergodic measure for irreducble WGM}, $f^{R}$ admits a unique ergodic invariant probability measure $\nu_{0}$ which is equivalent $\m$. We can define a measure $\tilde\mu$ on $M$ by
$$\tilde\mu=\sum_{j=0}^{\infty}f_{\ast}^{j}(\nu_{0}|\{R>j\}).$$
If $R\in L^{1}(m)$ and $f_{*}\m\ll \m$, then by standard results, $\mu=\frac{\tilde\mu}{\tilde\mu(M)}$ is the unique ergodic $f$-invariant probability measure such that $\mu\ll \m$ and $\mu(\Delta_{0})>0$.
\end{remark}

\section{Mixing measures for Tower maps}\label{Mixing measure for TM}

In this section, we first discuss the existence of an invariant probability measure for the tower map of an induced WGM map. Then, we provide conditions to ensure for this measure to be mixing.

Assume that $f^{R}: \Delta_{0}\rightarrow \Delta_{0}$ is an induced WGM map for $f: M\rightarrow M $.
We define a \emph{tower}
$$ \Delta=\{(x,\ell): x\in\Delta_{0} \mbox{ and } R(x)> \ell \geq 0 \}   $$

 and the tower map  $T: \Delta\rightarrow \Delta$ given by
$$ T(x,\ell)=\left\{
\begin{array}{ll}
	(x,\ell+1), & \mbox{ if } R(x)> \ell+1 \\
	(f^{R}(x),0), & \mbox{ if } R(x)=\ell+1
\end{array}
\right.     $$
Note that we can naturally identify the set $\{(x,0)\colon x\in \Delta_0\}\subset\Delta$  with $\Delta_{0}$, and the induced map $T^{R}:\Delta_{0}\rightarrow \Delta_{0}$ with $f^{R}$.
For each $\ell\geq0$, we define the \textit{$\ell$th level} of the tower 
$$\Delta_{\ell}=\{(x,\ell): x\in \Delta_0 \},$$
which is naturally identified with  $\{R>\ell\}\subset\Delta_{0}$. In view of this, we may extend the $\sigma$-algebra $\mathcal A$ and the reference measure $ \m $ on $ \Delta_{0} $ to a $\sigma$-algebra and a measure on $ \Delta $, that we still denote by $\mathcal A$ and $ \m $, respectively. We have
$$  
m(\Delta)=\int_{ \Delta_{0}}R\,dm,
$$
which implies that measure $ \m $ on $ \Delta $ is finite if and only if $ R\in L^1(\m) $.   
The countable partition $ \mathcal{P}_{0} $ of $ \Delta_{0} $ naturally induces an ($\m$ mod 0) partition of each level, that is, if $\mathcal{P}_{0}=\{\Delta_{0,i}\}_{i\in\mathbb N}$ is the partition of $ \Delta_{0} $, then $\{\Delta_{\ell,i}\}_{i \in \mathbb{N}}$, where $\Delta_{\ell,i}=\{(x,\ell)\in \Delta_{\ell}: (x,0)\in\Delta_{0,i}\} $ forms a partition of $\Delta_{\ell}.$
Collecting all these partitions, we obtain an ($\m$ mod 0) partition $ \eta $ of $ \Delta .$ 
For each $ n\geq 1 $, we set
$$ 
\eta_{n}=\bigvee_{i=0}^{n-1}T ^{-i}\eta.
$$ 
We can extend the separation time to $ \Delta\times\Delta $ in the following way:
if $ x, y \in \Delta_{\ell} $, then there are unique  $ x_{0}, y_{0} \in \Delta_{0} $ such that $x=T^{\ell}(x_{0}) $ and $ y=T^{\ell}(y_{0}) $, and in this case we set $ s(x, y)=s(x_{0}, y_{0}) $, otherwise set $ s(x, y)=0 $.
It is straightforward to check that $ J_{T} $ is
$$ 
J_{T}(x,\ell)=\left\{
\begin{array}{ll}
	1, & \mbox{ if } R(x)> \ell+1 \\
	J_{f^{R}(x)}, & \mbox{ if } R(x)=\ell+1.
\end{array}
\right.  
$$
Given $0< \beta <1$ we can similarly define the following spaces of densities for the tower: 
\begin{equation*}
	\mathcal{F}_{\beta}( \Delta ) = \{\varphi : \Delta\rightarrow \mathbb{R} \colon \exists C_{\varphi} > 0 :|\varphi(x)-\varphi(y)|\leq C_{\varphi}\beta^{s(x,y)}, \forall x,y \in \Delta\}
\end{equation*}
and
\begin{multline*}
	\mathcal{F}_{\beta}^{+}( \Delta ) = \{\varphi\in \mathcal{F}_{\beta}( \Delta ) \colon \exists C_{\varphi}' > 0 \, \text{ s.t. }\,\\ \varphi(x)>0 \text{ and } \left|\frac{\varphi(x)}{\varphi(y)}-1\right|\leq C_{\varphi}'\beta^{s(x,y)}, \forall x,y \in\omega \in \eta \}.
\end{multline*}
Assuming $ R\in L^{1}(m) $, we have $\mathcal{F}_{\beta}(\Delta)\subset L^{\infty}(m) \subset L^{1}(m)$. The next result is similar to~\cite[Theorem 1]{Y99} (see also~\cite[Theorem 3.24]{Alves Book}). We notice that in our setting we do not have a positive lower bound for $d\nu/d\m$.

\begin{theorem}\label{T1}
Let $T: \Delta\rightarrow \Delta$ be the tower map of an irreducible induced WGM map $f^R$ with $R\in L^{1}(m)$. If $\nu_{0}$ is the unique ergodic $f^{R}$-invariant probability measure such that $ \frac{d\nu_{0}}{dm}\in \mathcal{F}_{\beta}^{+}( \Delta_0 ) $, then	
$$\nu=\frac{1}{\sum_{j=0}^{\infty}\nu_{0}\{R>j\}}\sum_{j=0}^{\infty}T_{\ast}^{j}(\nu_{0}|\{R>j\})$$ 
is the unique ergodic $T$-invariant probability measure such that $\nu\ll m$ with $ \frac{d\nu}{dm}\in \mathcal{F}_{\beta}^{+}( \Delta )$, and there is $ C_{0}>0 $  such that $0<\frac{d\nu}{dm}\leq C_{0}.$
\end{theorem}

\begin{remark}\label{re:log prop on tower}
From Remark~\ref{RD1} and the fact that the measure $\nu$ in an upper level of the tower is a normalized copy of the  measure $\nu_0$ restricted to some subset of the basis $\Delta_0$, we conclude that there exists $C > 0$ such that for all $ x,y \in\omega \in \eta$
$$\left|\log\frac{\frac{d\nu}{dm}(x)}{\frac{d\nu}{dm}(y)}\right|\leq C\beta^{s(x,y)}.$$
\end{remark}
\medskip
In general, the ergodic $ T $-invariant probability measure $ \nu\ll m $ given by Theorem \ref{T1} is not mixing.
For instance, if $ \gcd\{R\}=d\geq2 $, then we have
$$ T^{-dn+1}(\Delta_{0})\cap \Delta_{0}=\emptyset, \text{ for } n\geq 1,$$
but $ \nu(\Delta_{0})>0$, which means that the measure $ \nu $ is not mixing. Hence we conclude that $ \gcd\{R\}=1 $ is a necessary condition for the measure $ \nu $ to be mixing.
If $f^{R} :\Delta_{0} \rightarrow \Delta_{0}$ is an induced WGM map that satisfy the full branch property then $ \gcd\{R\}=1 $ is a sufficient condition to conclude that $ \nu $ is mixing \cite{Y99}. However, in the absence of the full branch property, this condition is not enough to guarantee that the measure $ \nu $ is mixing.
 For instance, consider $ M=\Delta_{0}=[0,1] $, $\m$ to be the Lebesgue measure, $ \mathcal{P}_{0}=\{\omega_{1}, \omega_{2}, \omega_{3}\} $ an ($m$ mod 0) partition of $ \Delta_{0} $, with $\m(\omega_i)>0$, for $i=1,2,3$, and $f\colon\Delta_0\to\Delta_0$ such that
\[ f(\omega_{1})=\omega_{2}\cup\omega_{3}, \,\, f(\omega_{2})=\omega_{1} \,\,\text{ and }\,\, f(\omega_{3})=\omega_{1}.
\]
 We define $ R:\Delta_{0}\rightarrow \mathbb{N} $ as $ R(\omega_{i})=i$, $i=1,2,3$.
We have 
 $ (f^R)^2(\omega_{1})=\Delta_{0} $, $ (f^R)^2(\omega_{2})=\Delta_{0} $ and $  (f^R)^3(\omega_{3})=\Delta_{0}.$ This means that $ f^R $ is an aperiodic induced WGM map with $\gcd\{R\}=1$. We can see that for the tower map $T$ of $f^R$ we have 
 $ T^k(\omega_{1})\cap\omega_{1}=\emptyset,$ for $ k=2n+1$ and all $n\in\mathbb N$, but $ \nu(\omega_{1})>0 $, so the measure $ \nu $ is not mixing.

In order to ensure that the measure $ \nu $ is mixing with respect to a tower map of an aperiodic induced WGM map $f^R$, we also assume that $f^R$ has a coprime block. In fact, the coprime block is only used to get the measure $ \nu $ (hence $\mu$) to be mixing.

\begin{lemma}\label{doubly aperiodic lemma}
	Let  $T: \Delta\rightarrow \Delta$ be the tower map of an aperiodic induced WGM map $ f^R$ with a coprime block. For all $\omega \in \mathcal{P}_{0}^{n}$ there exist $ t_{0}=t_{0}(\omega)\in \mathbb{N}$ such that  $T^{t}(\omega)\cap\Delta_{0} \neq \emptyset, \forall t \geq t_{0}$.
\end{lemma}
\begin{proof}
	By the definition of a coprime block, we can find $ \alpha_{1}, \alpha_{2}, ...\alpha_{N}\in \mathcal{P}_0 $ such that, setting $ R(\alpha_i)=R_{i} $, we have $\gcd\{R_{i}\}_i=1$ and  
	$$
	f^{R}(\alpha_{i})\supseteq \alpha_{1}\cup\alpha_{2}\cup\cdots\cup\alpha_{N}\quad
	$$
for all $i=1,\ldots,N$. Since $\gcd\{R_{i}\}_{i}=1$, there exists $t_{0}^{\prime} \in \mathbb{N}$ such that for all $t' \geq t_{0}^{\prime}$ we can find integers $a_{1},..., a_{k} \geq 0$ for which $t'=a_{1} R_{1}+...+a_{k} R_{k}$, which implies that 
	\begin{equation}\label{doubly aperiodic eq}
	T^{t'}(\alpha_{1}\cup\alpha_{2}\cup...\cup\alpha_{N})\supseteq \alpha_{1}\cup\alpha_{2}\cup\cdots\cup\alpha_{N}, \text{ for all } t' \geq t_{0}^{\prime}.
\end{equation}
	By the aperiodicity of $ (f^R, \mathcal{P}_{0}) $, for $\omega \in \mathcal{P}_{0}^{n}$, there exists $ n_{0}\in \mathbb{N} $ such that $  (f^R)^{n}(\omega)\supseteq \alpha_{1}$ for all  $n\geq  n_{0}$.
Then, we can find 
\[\tilde{\omega}=\omega\cap(f^R)^{-1}(\omega_{1})\cap...\cap(f^R)^{-n_{0}+1}(\omega_{n_{0}-1})\in \mathcal{P}_0^{n_{0}}\]
  for some $ \omega_{1}, \omega_{2},...,\omega_{n_{0}-1}\in \mathcal{P}_{0} $ such that $(f^R)^{n_{0}}(\tilde{\omega})=\alpha_{1}$.
Taking 
\[ K_{\alpha_{1}}= R(\omega)+ R(\omega_{1})+...+R(\omega_{n_{0}-1}) \]
 we can easily ensure that  $T^{ K_{\alpha_{1}}}(\omega)\supseteq \alpha_{1}$.
By \eqref{doubly aperiodic eq} and the fact that \[T^{ K_{\alpha_{1}}+R_{1}}(\omega)\supseteq \alpha_{1}\cup\alpha_{2}\cup\cdots\cup\alpha_{N},\]
 we have $T^{t'}(T^{K_{\alpha_{1}}+R_{1}}(\omega))\supseteq \alpha_{1}\cup\alpha_{2}\cup\cdots\cup\alpha_{N}$, for all $ t' \geq t_{0}^{\prime}.$ Hence we conclude that 
 $T^{t}(\omega) \cap \Delta_{0} \neq \emptyset$, for all $t \geq t_{0}^{\prime}+K_{\alpha_{1}}+R_{1}.$
\end{proof}

\begin{lemma}\label{doubly aperiodic Corr}
Let $T: \Delta\rightarrow \Delta$ be the tower map of an aperiodic induced WGM map $f^R$ with a coprime block. For all $\omega \in \mathcal{P}_{0}^{n}$, $\epsilon_{0}>0$ and $ \ell_{0}\geq0 $ there exist $ t_{0}\in \mathbb{N}$, such that  for all $t \geq t_{0}$,
\begin{equation*}
	m(T^{t} (\omega)) \geq m\left(\bigcup_{j \leq \ell_{0}} \Delta_{j}\right)-\epsilon_{0}. 
\end{equation*}
\end{lemma}

\begin{proof}
	Let $ Z $ be a finite union of elements of $ \mathcal{P}_{0} $  such that $ m(Z)\geq m(\Delta_{0})-\frac{\epsilon_{0}}{\ell_{0}} $. 
	By aperiodicity of $ (f^R, \mathcal{P}_{0}) $, for  $\omega \in \mathcal{P}_{0}^{n}$, there exist $ n_{0}\in \mathbb{N} $ such that $$m\big((f^R)^n(\omega)\big)\geq m(Z),  \text{ for all } n\geq  n_{0}.$$
 As in the proof of Lemma ~\ref{doubly aperiodic lemma}, we have 
	\begin{equation}\label{d.aperiodic cor eq  }
	T^{t'}(T^{K_{\alpha_{1}}+R_{1}}(\omega))\supseteq \alpha_{1}\cup\alpha_{2}\cup\cdots\cup\alpha_{N}, \text{ for all }  t' \geq t_{0}'.
	\end{equation}
	Let $ \bar{\alpha} $ be an arbitrary element of $ \mathcal{P}_{0} $ such that $\bar{\alpha}\subset Z  $. For any $ \alpha_{j} \in \{ \alpha_{1}, \alpha_{2},...,\alpha_{N} \}  $, there exist $ k_{0}\in \mathbb{N} $ such that $(f^R)^{n}(\alpha_{j})\supseteq \bar{\alpha}$ for all  $n\geq  k_{0}$.
	As in the proof of Lemma ~\ref{doubly aperiodic lemma} we can find $ \tilde{\alpha_{j}}\in \mathcal{P}_{0}^{k_{0}}$, $K_{\bar{\alpha}}\in \mathbb{N}$ such that $(f^R)^{k_{0}}(\tilde{\alpha_{j}})=\bar{\alpha}$ and  $ T^{ K_{\bar{\alpha}}}(\alpha_{j})\supseteq \bar{\alpha}.$ 
	Since $T^{t'}(\alpha_{1}\cup\alpha_{2}\cup...\cup\alpha_{N})\supseteq \alpha_{1}\cup\alpha_{2}\cup...\cup\alpha_{N} $ for all $t' \geq t_{0}'$ and $T^{ K_{\bar{\alpha}}}(\alpha_{j})\supseteq \bar{\alpha}$, this implies that $ T^{t'+ K_{\bar{\alpha}}}(\alpha_{1}\cup\alpha_{2}\cup...\cup\alpha_{N})\supseteq \bar{\alpha} $ for all $t' \geq t_{0}'$ and hence we can find $t_1'$ such that  $ T^{t'}(\alpha_{1}\cup\alpha_{2}\cup...\cup\alpha_{N})\supseteq Z $ for all $t' \geq t_{1}'$, which, together with \eqref{d.aperiodic cor eq }, implies that we  can find $t_2'$ such that $ T^{t'}(\omega)\supseteq Z $ for all $t' \geq t_{2}'$.
		Hence by using the above fact and Lemma \ref{doubly aperiodic lemma}, we can find  $ t_0> 0 $ such that for all $t\geq t_{0}$,
	\begin{equation*}
		m(T^{t} (\omega))\geq\sum_{j\leq\ell_{0}}\left(\m(\Delta_{j})-\frac{\epsilon_{0}}{\ell_{0}}\right)= \m\left(\bigcup_{j \leq \ell_{0}} \Delta_{j}\right)-\epsilon_{0}. 
	\end{equation*}
\end{proof}

\begin{theorem}
	Let  $T: \Delta\rightarrow \Delta$ be the tower map of an aperiodic induced WGM map $ f^R $ with a coprime block and $R\in L^{1}(m)$. Then the $T$-invariant probability measure $\nu$ is exact. 
\end{theorem}
\begin{proof}
	Take an arbitrary $\epsilon>0$ and let $ C_{0}>0 $ be given by Theorem~\ref{T1}. Since $ m(\Delta)< \infty $, we can find $\ell_{0}\geq0$ such that
	\begin{equation}\label{Y1}
		m\left(\Delta\backslash\bigcup_{j \leq \ell_{0}} \Delta_{j}\right)<\frac{\epsilon}{3C_{0}}.
	\end{equation}
	For $\epsilon_0=\epsilon /3C_0$ we choose $t_0$ given by Lemma~\ref{doubly aperiodic Corr}. Since $T$ is non-singular with respect to $m$, taking $C>0$ to be the constant given by~\emph{\ref{lem:WGMpropii})} in Lemma~\ref{lem:WGMprop}, we can find $\delta>0$ such that, for all $B \in \mathcal{A}$ 
	\begin{equation}\label{Y02}
		m(B)<Cm(\Delta_{0})\delta \Rightarrow m(T^{t_0}(B))< \frac{\epsilon}{3C_{0}}.
	\end{equation}
		Consider now $A \in \bigcap_{n \in \mathbb{N}} T^{-n} (\mathcal{A})$, with $\nu(A)>0$.
	Since $\nu\ll m$, we have $m(A)>0$, and then there exists $s\geq 0$ such that $m(T^{s}(A) \cap \Delta_{0})>0.$
	Applying \cite[Corollary 2.3]{Alves Book} to the WGM map $ f^R: \Delta_{0}\rightarrow \Delta_{0} $ and the set   $T^{s}(A) \cap \Delta_{0}$, we find $ n\in \mathbb{N} $ and  $\omega \in \mathcal{P}_{0}^{n} $ such that
	\begin{equation*}
		\frac{m(\omega\backslash T^{s}(A) \cap \Delta_{0})}{m(\omega)}=\frac{m(\omega\backslash T^{s}(A))}{m(\omega)}<\delta .
	\end{equation*}
	Applying~\emph{\ref{lem:WGMpropii})} of~Lemma~\ref{lem:WGMprop} to $\omega \in \mathcal{P}_{0}^{n}$ we get
	
	$$
	\begin{aligned}
		& \frac{\m( (f^R)^{n}(\omega \backslash T^{s}(A)))}{\m( (f^R)^{n}(\omega))} \leq C \frac{\m(\omega \backslash T^{s}(A))}{\m(\omega)}< C \delta.
	\end{aligned}
	$$
	From the above inequality, we have 
	\begin{equation}\label{Y3} 
		\m( (f^R)^{n}(\omega \backslash T^{s}(A)))<C \delta \m( (f^R)^{n}(\omega))\leq C \delta m(\Delta_{0}).
	\end{equation} 
		Since $\omega \in \mathcal{P}_{0}^{n}$, the points in $\omega \in \mathcal{P}_{0}^{n}$ have the same first $n$-recurrence times to the base of the tower. This means that there are natural numbers $R_{1},\ldots, R_{n} \in R(\Delta_{0})$ such that
	$$
	(f^R)^{n}|_{ \omega}=T^{R_{1}+\cdots+R_{n}}|_{\omega} .
	$$
		Set $ r=R_{1}+\cdots+R_{n} $. It follows from (\ref{Y3}) that 
	\begin{equation*}
		m(T^{r} (\omega) \backslash T^{r+s}(A))< C \delta m(\Delta_{0}).
	\end{equation*}
By~\eqref{Y02} we get
\begin{equation}\label{Y4} 
m(T^{t_0+r}(\omega) \backslash T^{t_0+r+s}(A))<\frac{\epsilon}{3C_{0}}.
\end{equation}
By Lemma~\ref{doubly aperiodic Corr}, we have
	
	\begin{equation}\label{eq:bottomfromLemma}
		\m(\bigcup_{j \leq \ell_{0}} \Delta_{j} \backslash T^{t_0+r+s}(A)) 
 <\m(T^{t_0+r}(\omega) \backslash T^{t_0+r+s}(A))+\frac{\epsilon}{3C_{0}} \\
		\end{equation}	
	
Using \eqref{Y1}, \eqref{eq:bottomfromLemma} and~\eqref {Y4} we have
	$$
	\begin{aligned}
		\m(\Delta \backslash T^{t_0+r+s}(A))& 
		=\m((\Delta \backslash \bigcup_{j \leq \ell_{0}} \Delta_{j}) \backslash T^{t_0+r+s}(A))+\m(\bigcup_{j \leq \ell_{0}} \Delta_{j} \backslash T^{t_0+r+s}(A)) \\
		& <\frac{\epsilon}{3C_{0}}+m(T^{t_0+r}(\omega) \backslash T^{t_0+r+s}(A))+\frac{\epsilon}{3C_{0}} \\
		& <\frac{\epsilon}{3C_{0}}+\frac{\epsilon}{3C_{0}}+\frac{\epsilon}{3C_{0}}=\frac{\epsilon}{C_{0}}.
	\end{aligned}$$
	This implies $ \nu(\Delta \backslash T^{t_0+r+s} (A))<\epsilon $.
	Since $A \in \bigcap_{n \in \mathbb{N}} T^{-n} (\mathcal{A})$, then $A=T^{-(t_0+r+s)} (A^{\prime})$, for some $A^{\prime} \in \mathcal{A}$ with $ T^{t_0+r+s}(A)=A^{\prime}.$ 
		Since $\nu(A^{\prime})=\nu(T^{t_0+r+s}(A))>1-\epsilon$, using the invariance of $ \nu $ we have
	$$
	\nu(A)=\nu(A^{\prime})=\nu\left(T^{-(t_0+r+s)} (A')\right)>1-\epsilon 
	$$
	Since $ \epsilon $ is arbitrary, we have $ \nu(A)=1 $, so $ \nu $ is exact.
\end{proof}

\section{Decay of Correlations for Tower Maps}\label{se:dc on tower}
Let us start by defining a measurable semi-conjugacy  $\pi : \Delta \rightarrow M$ between the tower map $T$ and the map $f$, by $\pi(x,\ell)=f^{\ell}(x)$. We have  $\pi\circ T = f\circ \pi$ and $ \pi_{*} \nu$ coincides with the $f$-invariant measure $\mu$ given by Remark \ref{gives ergodic measure for f}. It is an immediate consequence that the measure $\mu$ is mixing and for  all $\varphi, \psi : M\rightarrow \mathbb{R}$,
\begin{equation*}
	\Cor_{\mu}(\varphi, \psi \circ f^{n}) = \Cor_{\nu}(\varphi \circ \pi, \psi \circ \pi \circ T^{n}).
\end{equation*}
Note that to obtain a decay of correlations for $(f, \mu)$, it therefore suffices to obtain a decay of correlations for $ (T, \nu) $. 

Given observables $\varphi \in \mathcal{H}_{\eta} $ and $\psi \in L^\infty(M,\m)$, we need that the lifted observables belong to some suitable spaces, namely
$ \varphi \circ \pi \in \mathcal{F}_{\beta}( \Delta ) $ and $\psi \circ \pi\in L^\infty(\Delta,\m)$. Obviously, if $\psi \in L^\infty(M,\m)$ then $\psi \circ \pi\in L^\infty(\Delta,\m)$. In order to prove $ \varphi \circ \pi \in \mathcal{F}_{\beta}( \Delta ) $ for $\varphi \in \mathcal{H}_{\eta}$ we have the following result.
\begin{lemma}\cite[Lemma 3.53]{Alves Book}\label{lemma: phi circ pi}
	If $T: \Delta \rightarrow \Delta$ is the tower map of an induced WGM expanding map $f^{R} :\Delta_{0} \rightarrow \Delta_{0}$, then 
	$	\varphi \in \mathcal{H}_{\eta}   \Rightarrow \varphi \circ \pi \in \mathcal{F}_{\beta^{\eta}}(\Delta).
	$
\end{lemma}
Given $\varphi\in L^\infty(\Delta)$ we define
 $$
 \varphi^{\ast}=\frac{1}{\int(\varphi+2||\varphi||_{\infty}+1)d\nu}(\varphi+2||\varphi||_{\infty}+1).
 $$
\begin{lemma}\cite[Lemma 3.49]{Alves Book}\label{D.Correlation concection of T and convergence of equlibrium}
For all $\varphi\in\mathcal{F}_{\beta}(\Delta)$, with $ \varphi\neq 0 $, we have
\begin{itemize}
\item[i)] $\frac{1}{3}\leq \varphi^{\ast}\leq 3$ and $\varphi^{\ast}\in  \mathcal{F}_{\beta}^{+}(\Delta)$;
\item[ii)]  $\Cor_{\nu}(\varphi,\psi \circ T^{n})\leq3(||\varphi||_{\infty}+1)||\psi||_{\infty}|T_{\ast}^{n}\lambda-\nu|$, for all $\psi\in L^{\infty}(m)$, where $\lambda$ is the probability measure on $\Delta$ such that  $\frac{d\lambda}{dm}=\varphi^{\ast}\frac{d\nu}{dm}.$
\end{itemize}
\end{lemma}

From $ii)$ of Lemma \ref{D.Correlation concection of T and convergence of equlibrium}  to obtain decay of correlation for the tower map it is enough to estimate $ |T_{*}^n\lambda-\nu| $. For that we have the following.

\begin{proposition}\label{Convergence to equilibrium }
Let $ T:\Delta\rightarrow \Delta $ be the tower map of an aperiodic induced WGM map $ f^{R}$ with a coprime block and $ R\in L^{1}(m) $. If $ \nu $ is the unique mixing $ T$-invariant probability measure such that $ \frac{d\nu}{dm}\in \mathcal{F}_{\beta}^{+}( \Delta ) $, then
\begin{enumerate}
\item if $m\{R>n\} \leq C n^{-a} $ for some $C>0$ and $a>1$, then for any probability measure $ \lambda $ with $ \frac{d\lambda}{dm}\in \mathcal{F}_{\beta}^{+}( \Delta ) $, there exists $ C^{\prime} $ such that 
\begin{equation*}
|T_{*}^{n}\lambda-\nu|\leq C^{\prime}n^{-a+1}.
\end{equation*}
\item if $m\{ R >n \} \leq C e^{-cn^{a}}$ for some $C,c>0$ and $ 0 < a \leq 1$, given $0< \beta <1$ there is $c{'}>0$ such that, for any probability measure $ \lambda $ with $ \frac{d\lambda}{dm}\in \mathcal{F}_{\beta}^{+}( \Delta ) $, there is some $C^{'} > 0$ such that
\begin{center}
	$|T_{*}^{n}\lambda-\nu| \leq C^{'} e^{-c{'}n^{a}}$.
\end{center}
\end{enumerate}
\end{proposition}
We remark that in this result the  coprime block and the aperiodicity are only used to ensure that the $T$-invariant probability measure $\nu$ is unique, mixing and that $d\nu/d\m>0$. Now, in order to apply Proposition~\ref{Convergence to equilibrium }, given $\varphi \in \mathcal F_{\beta}(\Delta)$	we need to check the necessary regularity of 
$\frac{d\lambda}{dm}=\varphi^{\ast}\frac{d\nu}{dm}$.

\begin{lemma}\label{lemma:phi*}
If $\varphi \in \mathcal F_{\beta}(\Delta)$, then
$\varphi^{\ast}\frac{d\nu}{dm}\in \mathcal{F}_{\beta}^{+}(\Delta)$.
\end{lemma}

\begin{proof}
Set $\rho=\frac{d\nu}{dm}$. We notice that $\int(\varphi+2||\varphi||_{\infty}+1)d\nu\geq1$, and since $0<\rho\leq C_{0}$, we get
\begin{equation*}
		\begin{aligned}
			|\varphi^{\ast}(x)\rho(x)-\varphi^{\ast}(y)\rho(y)|&\leq|\varphi^{\ast}(x)(\rho(x)-\rho(y))|+|\rho(y)(\varphi^{\ast}(x)-\varphi^{\ast}(y))|\\
			&\leq 3|\rho(x)-\rho(y)|+C_{0}|\varphi^{\ast}(x)-\varphi^{\ast}(y)|\\
			&=3|\rho(x)-\rho(y)|+\frac{C_{0}}{\int(\varphi+2||\varphi||_{\infty}+1)d\nu}|\varphi(x)-\varphi(y)|\\
			&\leq 3|\rho(x)-\rho(y)|+C_{0}|\varphi(x)-\varphi(y)|\\
			&\leq3C_\rho \beta^{s(x,y)}+C_{0}C_{\varphi}\beta^{s(x,y)}\\
			&=C_{\varphi^*\rho}\beta^{s(x,y)},
		\end{aligned}
	\end{equation*} 
	where $C_{\varphi^*\rho}=3C_\rho+C_{0}C_{\varphi}.$
	Thus $\varphi^{\ast}\rho\in \mathcal{F}_{\beta}(\Delta)$. 
		From Lemma \ref{D.Correlation concection of T and convergence of equlibrium} we have
	\begin{equation}
		\begin{aligned}
			\left|\frac{\varphi^{\ast}(x)}{\varphi^{\ast}(y)}-1\right|&=\frac{1}{\varphi^{\ast}(y)}\left|\varphi^{\ast}(x)-\varphi^{\ast}(y)\right|\\
			&\leq3\frac{  \left|\varphi(x)-\varphi(y)\right|  }{\int(\varphi+2||\varphi||_{\infty}+1)d\nu}\\
			&\leq 3 \left|\varphi(x)-\varphi(y)\right|\\
			&\leq3C_{\varphi}\beta^{s(x,y)}.\label{A.claim eq}
		\end{aligned}
	\end{equation}
		Since $\frac{1}{9}\leq \frac{\varphi^{\ast}(x)}{\varphi^{\ast}(y)}\leq 9 $ for all $ x, y\in \Delta $, then there exists $ K_{1}^{\prime}>0 $ such that 
	\begin{equation}\label{A.V.L eq for log-densisity}
		\left|\log\frac{\varphi^{\ast}(x)}{\varphi^{\ast}(y)} \right|\leq K_{1}^{\prime}\left|\frac{\varphi^{\ast}(x)}{\varphi^{\ast}(y)}-1\right|.
	\end{equation}
	It follows from   \eqref{A.claim eq},\eqref{A.V.L eq for log-densisity} and Remark~\ref{re:log prop on tower} that for all $ x,y\in \omega\in \eta $,
	\begin{equation*}
		\begin{aligned}
			\left|\log\frac{\varphi^{\ast}(x)\rho(x)}{\varphi^{\ast}(y)\rho(y)} \right|&\leq\left|\log\frac{\varphi^{\ast}(x)}{\varphi^{\ast}(y)} \right|+\left|\log\frac{\rho(x)}{\rho(y)} \right|\\
			&\leq 3K_{1}^{\prime} C_{\varphi}\beta^{s(x,y)}+C\beta^{s(x,y)}\\
			&=(3K_{1}^{\prime} C_{\varphi}+C)\beta^{s(x,y)}.
		\end{aligned}
	\end{equation*}
		Since $ 0<\frac{\varphi^{\ast}(x) \rho(x)}{\varphi^{\ast}(y) \rho(y)}\leq 9e^{C} $, for all $ x, y\in \omega\in \eta $, then we also have some uniform constant $ K_{2}^{\prime}>0 $ such that 
	\begin{equation*}
		\left|\frac{\varphi^{\ast}(x)\rho(x)}{\varphi^{\ast}(y)\rho(y)}-1 \right|\leq K_{2}^{\prime}\left|\log\frac{\varphi^{\ast}(x)\rho(x)}{\varphi^{\ast}(y)\rho(y)} \right|\leq K_{2}^{\prime}(3K_{1}^{\prime} C_{\varphi}+C)\beta^{s(x,y)}.
	\end{equation*}
Hence $\varphi^{\ast}\rho\in \mathcal{F}^{+}_{\beta}(\Delta).$
\end{proof}

\begin{remark}
From the previous proof  we can see that there exists $C_{\varphi^*\rho}''>0$ such that for all  $ x,y\in \omega\in \eta $,
	\begin{equation*}
				\left|\log\frac{\varphi^{\ast}(x)\rho(x)}{\varphi^{\ast}(y)\rho(y)} \right|\leq C_{\varphi^*\rho}''  \beta^{s(x,y)}.
			\end{equation*}
\end{remark}
In the remaining of this section we prove Proposition \ref{Convergence to equilibrium }, from which we directly obtain our main Theorem~~\ref{teo:main}.

\medskip

Let $ \lambda_{1} $ and $ \lambda_{2} $ be probability measures on $ \Delta $ with $\varphi_{1}= \frac{d\lambda_{1}}{dm}, \varphi_{2}=\frac{d\lambda_{2}}{dm} \in  \mathcal{F}^{+}_{\beta}( \Delta )$. 
We start with some bounded distortion results for the tower map $ T:\Delta\rightarrow \Delta$.
\begin{lemma}\label{L20}\cite[Lemma 3.30]{Alves Book}
	There exists $ \bar{C_{T}}>0 $ such that for all $ n\geq1, $ $\omega \in \eta_{n}$ and $ x, y\in \omega, $ 
	$$
	\log\frac{J_{T^{n}}(x)}{J_{T^{n}}(y)}\leq \bar{C_{T}}\beta^{s(T^{n}(x), T^{n}(y))}.
	$$
\end{lemma}
The next Lemma is the analogous of ~\cite[Lemma 3.31]{Alves Book}. Instead of taking $ x,y \in \Delta_{0} $, we pick $ x, y $ in the image of the element of partition.
\begin{lemma}\label{L26}
	There exists $ C_{1}>0$ depending on $ C_{\varphi_{1}}' $ such that for all $ n\geq 1 $, $ \omega \in \eta_{n}$ and $ x, y \in T^{n}(\omega)$, we have 
	\begin{equation*}
		\frac{dT_{*}^{n}(\lambda_{1}|\omega)}{dm}(x)\Big/ \frac{dT_{*}^{n}(\lambda_{1}|\omega)}{dm}(y) \leq   C_{1}.		
	\end{equation*}
	Moreover, the dependence of $  C_{1} $ on $ C_{\varphi_{1}}' $ can be removed if the number of visits of $\omega$ to $ \Delta_{0} $ up to $ n $ is sufficiently large. 
	
\end{lemma}
\begin{proof}
	Let $ x_{0}, y_{0} \in \omega $ be such that $ T^{n}(x_{0})=x$ and $T^{n}(y_{0})=y.$ We may write
	\begin{equation}\label{E25}
		\frac{dT_{*}^{n}(\lambda_{1}|\omega)}{dm}(x)\Big/ \frac{dT_{*}^{n}(\lambda_{1}|\omega)}{dm}(y) =\frac{J_{T^{n}}( y_{0})}{J_{T^{n}}( x_{0})} \frac{\varphi_{1}(x_{0})}{\varphi_{1}(y_{0})}.		
	\end{equation}
		By Lemma \ref{L20}, there exists some $ C_{T}^{\prime}>0 $ such that 
	\begin{equation}\label{E26}
		\frac{J_{T^{n}}( y_{0})}{J_{T^{n}}( x_{0})}\leq C_{T}^{\prime}.
	\end{equation}
	Considering $ i $ the number of visits of $ \omega $ to $ \Delta_{0} $ until $ n $, we have $ s(x_{0},y_{0})\geq i $. Since $ \varphi_{1} \in \mathcal{F}^{+}_{\beta}( \Delta )$ we may write
	\begin{equation}\label{E27}
		\frac{\varphi_{1}(x_{0})}{\varphi_{1}(y_{0})}=1+\left|\frac{\varphi_{1}(x_{0})}{\varphi_{1}(y_{0})}-1\right|\leq 1+C_{\varphi_{1}}'\beta^{i}.
	\end{equation}
	From (\ref{E25}), (\ref{E26}) and (\ref{E27}), we get 
	\begin{center}
		$\frac{dT_{*}^{n}(\lambda_{1}|\omega)}{dm}(x)\Big/ \frac{dT_{*}^{n}(\lambda_{1}|\omega)}{dm}(y) \leq   C_{1},$		
	\end{center}
	where $   C_{1}= C_{T}^{\prime}(1+C_{\varphi_{1}}'\beta^{i}) $. Note that the constant $ C_{T}^{\prime} $ does not depend on $ \varphi_{1} $. So the dependence of $   C_{1} $ on $ \varphi_{1} $ can be removed if we take $ i $ sufficiently large. 
\end{proof}
\begin{remark}
	If we replace $ \lambda_{1} $ with $ \lambda_{2} $ in Lemma \ref{L26}, then the result remain true with $C_{1}=C_{T}^{\prime}(1+C_{\varphi_{2}}'\beta^{i}).$
\end{remark}

Let $P= \lambda_{1}\times\lambda_{2} $ be the product measure on $ \Delta\times\Delta .$ 
We consider the product transformation $ T\times T :\Delta\times\Delta\rightarrow \Delta\times\Delta $, and let $ \pi_{1}, \pi_{2} :\Delta\times\Delta\rightarrow \Delta $ be the projections onto the first and second coordinates, respectively. 
 We use $ \eta\times\eta $ to denote the product partition of $ \Delta\times\Delta $ and $ (\eta\times\eta)_{n}=\bigvee_{i=0}^{n-1}(T\times T) ^{-i}(\eta\times\eta). $
Note that 
\begin{equation}\label{E32}
	T^{n}\circ\pi_{1}=\pi_{1}\circ(T\times T) ^{n} \text{ and } T^{n}\circ\pi_{2} = \pi_{2}\circ(T\times T) ^{n}.
\end{equation}
Let $ \hat{R}: \Delta\rightarrow \mathbb{N} $ be the return time to the base, defined for $ x\in \Delta $ by
$$ \hat{R}(x)=\min\{n\geq0 : T^{n}(x)\in \Delta_{0} \}  .$$
Note that 
\begin{equation}\label{E29}
	m\{\hat{R}>n\}=\sum_{\ell >n}m(\Delta_{\ell})=\sum_{\ell >n}m\{R>n\}.
\end{equation}
From W\ref{def:WGM:Long})  there exists $ \mathcal{P}_{0}^{\prime}=\{\omega_{1}, \omega_{2}, \ldots,\omega_{k} \}\subset \mathcal{P}_{0} $ such that 
$$f^{R}(\omega)\cap  \mathcal{P}_{0}^{\prime}\neq \emptyset, \text{ for all } \omega\in \mathcal{P}_{0}. $$ 
This implies that for any $ \omega\in \mathcal{P}_{0} $, we have
\begin{equation}\label{E46}
	f^{R}(\omega)\supseteq \omega_{i} \text{ for some } \omega_{i}\in \mathcal{P}_{0}^{\prime}.
\end{equation} 
Since the measure $ \nu \ll m $ is mixing, for given $ \epsilon>0 $ and any $ \omega_{i}\in \mathcal{P}_{0}^{\prime}$ there exists $ n_{i}= n_{i}(\omega_{i})\in \mathbb{N}$ such that 
	$\nu(T^{-n}(\Delta_{0})\cap \omega_{i})\geq \nu(\Delta_{0})\nu(\omega_{i})-\epsilon>0$ for $ n\geq  n_{i}.$
As $ \frac{d\nu}{dm}\leq C_{0}, $ then for  $ M_{i}=\frac{1}{C_{0}}(\nu(\Delta_{0})\nu(\omega_{i})-\epsilon) $ we have
\begin{equation*}
	m(T^{-n}(\Delta_{0})\cap \omega_{i})\geq M_{i} \text{ for }  n\geq  n_{i}. 
\end{equation*}
Let $ n_{0}=\max\{n_{1}, n_{2},...,n_{k}\} $ and $ \gamma_{0}=\min\{ M_{1}, M_{2},...,M_{k}\} $. Then for all  $ 1\leq i\leq k $  we have
\begin{equation}\label{30}
	m(T^{-n}(\Delta_{0})\cap \omega_{i})\geq \gamma_{0} \text{ for }  n\geq  n_{0}. 
\end{equation}
It follows from (\ref{E46}) and (\ref{30}) that for any $ \omega\in \mathcal{P}_{0} $, we have
\begin{equation}\label{E47}
	m(T^{-n}(\Delta_{0})\cap f^{R}(\omega))\geq \gamma_{0} \text{ for }  n\geq  n_{0}. 
\end{equation}
Let us now introduce a sequence of \emph{stopping times} $ 0=\tau_{0}< \tau_{1}< \tau_{2}... $ on $ \Delta\times\Delta $,
\begin{equation*}
\begin{array}{lcl}
 \tau_{1}&=&n_{0}+\hat{R}\circ T^{ n_{0}}\circ\pi_{1}\\
\tau_{2}&=&n_{0}+\tau_{1}+\hat{R}\circ T^{ n_{0}+\tau_{1}}\circ\pi_{2}\\ 
\tau_{3}&=&n_{0}+\tau_{2}+\hat{R}\circ T^{ n_{0}+\tau_{2}}\circ\pi_{1}\\
\tau_{4}&=&n_{0}+\tau_{3}+\hat{R}\circ T^{ n_{0}+\tau_{3}}\circ\pi_{2}\\
&\vdots&
\end{array}
\end{equation*}
We define the simultaneous return to the base $ S:\Delta\times\Delta\rightarrow \mathbb{N} $ by
\begin{equation}\label{E30}
	S(x, y)=\min_{i\geq 2}\{\tau_{i}(x,y): (T^{\tau_{i}(x, y), }(x),T^{\tau_{i}(x, y), }(y))\in \Delta_{0}\times\Delta_{0}\},
\end{equation}
which is well defined $ m\times m $ almost everywhere.
Let $\xi_{0}<\xi_{1}<\xi_{2}<\xi_{3}...$ be an increasing sequence of partitions on $ \Delta\times\Delta $ defined as follows. As usual, given a partition $ \xi $, we denote $\xi(x)$  the element
of $ \xi $ containing $ x $.
First we take $\xi_{0}= \eta\times\eta$.
Now we describe the general inductive step in the construction of partitions $ \xi_{k} $. Assume that $ \xi_{j} $ has been constructed for all $ j<k $. The definition of $ \xi_{k} $ depends on whether $ k $ is odd or even. For definiteness we assume that $ k $ is odd. The construction for $ k $ even is the same apart from the change in the role of the first and second components.   
We let $ \xi_{k}=\{\xi_{k}(\overline{x}): \overline{x} \in \Delta\times\Delta \}$, where
\begin{equation*}
	\xi_{k}(\overline{x})= \bigvee_{i=0}^{\tau_{k}(\overline{x})-1}(T^{-i}(\eta))(x)\times\pi_{2}(\xi_{k-1}(\overline{x})).
\end{equation*}
 From the construction of above sequence of partitions $\xi_{k}$, we have the following properties:
\begin{itemize}
	\item  the functions $ \tau_{1},\tau_{2},...,\tau_{k} $ are constant on elements of $\xi_{k}$,
	\item the sets $ \{S=\tau_{k-1}\} $ and $ \{S>\tau_{k-1}\}$ can be written as union of elements of  $\xi_{k}.$
\end{itemize}
\begin{remark}\label{R2}
For any $\Gamma\in\xi_{k} $, we can find some $ \omega_{1}, \omega_{2}\in \mathcal{P}_{0} $ such that
\begin{itemize}
\item  if $k$ is odd,  $ T^{\tau_{k}}(\pi_{1}(\Gamma))=f^{R}(\omega_{1})\subset\Delta_{0} $ and $ T^{\tau_{k-1}}(\pi_{2}(\Gamma))=f^{R}(\omega_{2})\subset\Delta_{0}$;
\item  if $k$ is even,  $ T^{\tau_{k-1}}(\pi_{1}(\Gamma))=f^{R}(\omega_{1})\subset\Delta_{0} $ and $ T^{\tau_{k}}(\pi_{2}(\Gamma))=f^{R}(\omega_{2})\subset\Delta_{0}.$
\end{itemize}
\end{remark}
The following lemmas, which are similar of \cite[Lemma 1 and Lemma 2]{Y99}, are crucial to the proof of Proposition \ref{Convergence to equilibrium }. The key difference relies on the absence of the full branch property. More precisely, they will be helpful for estimating $P(S>n)$ in Section~\ref{Polynomial estimate} and Section~\ref{Exponential estimates}.

\begin{lemma}\label{L27}
	There exists $ \epsilon_{0}>0$, depending on $ C_{\varphi_1}', C_{\varphi_2}' $,  such that for all $ k\geq 1 $ and $ \Gamma \in\xi_{k}  $ with $ S|_\Gamma>\tau_{k-1}$, we have 
	\begin{equation*}
		P(S=\tau_{k} | \Gamma)\geq \epsilon_{0}.
	\end{equation*}
Moreover, the dependence of $\epsilon_{0}$ on $ C_{\varphi_1}', C_{\varphi_2}' $ can be removed if we take $ k $ sufficiently large. 

\end{lemma}
\begin{proof}
	Consider first $ k\geq 1$ is odd. For $\Gamma=\omega \times \omega^{\prime} \in \xi_k$, we have $T^{\tau_{k}}$ mapping $\omega$ bijectively onto a union of elements $\eta|_{\Delta_0}$ and $T^{\tau_{k-1}}$ mapping $\omega^{\prime}$ bijectively onto a union of elements $\eta|_{\Delta_0}$.
	Let $\tilde{\omega}=\{z\in\omega^{\prime}: T^{\tau_{k}}(z)\in (\Delta_{0})  \}$. We may write
\begin{equation*}
		P\left(S=\tau_{k} \mid \Gamma\right) =\frac{\lambda_{2}(\tilde{\omega})}{\lambda_{2}(\omega^{\prime})}=\frac{T_*^{\tau_{k-1}}\left(\lambda_{2}\mid\omega^{\prime}\right)\left(T^{-(\tau_{k}-\tau_{k-1})}(\Delta_{0})\cap T^{\tau_{k-1}}(\omega^{\prime})\right)}{T_*^{\tau_{k-1}}\left(\lambda_{2}\mid\omega^{\prime}\right)\left( T^{\tau_{k-1}}(\omega^{\prime})\right)}.
	\end{equation*}
	By Lemma \ref{L26}, we get
	$$  
	P\left(S=\tau_{k} \mid \Gamma\right)\geq  \frac{\m\left(T^{-(\tau_{k}-\tau_{k-1})}(\Delta_{0})\cap T^{\tau_{k-1}}(\omega^{\prime})\right)}{C_{1}^{2}\, \m\left( T^{\tau_{k-1}}(\omega^{\prime})\right)}\geq  \frac{\m\left(T^{-(\tau_{k}-\tau_{k-1})}(\Delta_{0})\cap T^{\tau_{k-1}}(\omega^{\prime})\right)}{C_{1}^{2}\, \m(\Delta_0)}.
	$$
By definition $\tau_{k}-\tau_{k-1}\geq n_{0} $ and, from Remark~\ref{R2}, $ T^{\tau_{k-1}}(\omega^{\prime})=f^{R}(\omega_{2}) $ for some $ \omega_{2}\in \mathcal{P}_{0} $. From (\ref{E47}) we get $ \m\left(T^{-(\tau_{k}-\tau_{k-1})}(\Delta_{0})\cap T^{\tau_{k-1}}(\omega^{\prime})\right)\geq \gamma_{0} $. Thus we conclude that
\begin{equation*}
P(S=\tau_{k} | \Gamma)\geq \epsilon_{0},
\end{equation*}
with $\epsilon_{0}= \gamma_{0}/C_{1}^{2}\m(\Delta_0)$.  The case $ k $ even can be proved similarly. Moreover, the dependence of $C_1$ (hence $\epsilon_0$) on $ C_{\varphi_1}', C_{\varphi_2}' $ can be removed if we take $ k $ large enough. 
\end{proof}
In the proof of next lemma, we will use the fact that there exists some constant $ M_0>0 $ such that 
\begin{equation}\label{E28}
	\frac{dT_{*}^{n}m}{dm}\leq M_0 \text{ for all } n\geq 1.
\end{equation}

\begin{lemma}\label{L28}
	There exists $ C_{2}$, depending on $ C_{\varphi_{1}}', C_{\varphi_{2}}' $, such that for all $n, k\geq 0 $ and for all $ \Gamma \in\xi_{k} $, we have 
	\begin{equation*}
		P(\tau_{k+1}- \tau_{k}>n+n_{0} | \Gamma)\leq C_{2}m\{\hat{R}>n\}.
	\end{equation*}
Moreover, the dependence of $C_{2}$ on $ C_{\varphi_{1}}, C_{\varphi_{2}} $ can be removed if we take $ k $ sufficiently large.

\end{lemma}
\begin{proof}
	First we consider $k=0$. By definition, $\tau_{k+1}- \tau_{k}=n_{0}+\hat{R}\circ T^{\tau_{k}+n_{0}}  $,  then we may write for all $n \geq 0$
	$$
	\begin{aligned}
		P\left\{\tau_1>n_0+n\right\}&=P\left(\{\hat{R}\circ T^{n_{0}}>n\}\times\Delta\right)\\
		&=\left(T_*^{n_0} \lambda_{1}\right)\{\hat{R}>n\} \\
		& \leq\|\varphi_{1}\|_{\infty}\left(T_*^{n_0} m\right)\{\hat{R}>n\}\\
		& \leq M_0\|\varphi_{1}\|_{\infty} m\{\hat{R}>n\}.
	\end{aligned}
	$$
	Now consider $k \geq 1$ is  odd. Since for $\Gamma=\omega \times \omega^{\prime} \in \xi_k$, we have $T^{\tau_{k-1}}(\omega)$ contained in some element of $\eta$ and $T^{\tau_{k-1}}$ mapping $\omega^{\prime}$ bijectively onto a union of elements in $\eta|_{\Delta_0}$. This implies that $\omega^{\prime} \in \eta_{\tau_{k-1}}$ and
	\begin{equation}\label{E42}
		\frac{1}{P(\Gamma)} T_*^{\tau_{k-1}} {\pi_2}_*(P \mid \Gamma)=\frac{ T_*^{\tau_{k-1}}\left(\lambda_{2} \mid \omega^{\prime}\right)}{T_*^{\tau_{k-1}}\left(\lambda_{2} \mid \omega^{\prime}\right)(T^{\tau_{k-1}}(\omega^{\prime}))}.
	\end{equation}
	Now,
	\begin{equation}\label{E43}
		\begin{aligned}
			P(\tau_{k+1}-\tau_k>n_0+n \mid \Gamma) &=P(\hat{R} \circ T^{\tau_k+n_0} \circ \pi_2>n \mid \Gamma) \\
			&=\frac{1}{P(\Gamma)}(P \mid \Gamma)\{\hat{R} \circ T^{\tau_k+n_0} \circ \pi_2>n\} \\
			&=\frac{1}{P(\Gamma)} T_*^{\tau_k+n_0} {\pi_2}_*(P \mid \Gamma)\{\hat{R}>n\} \\
			&=T_*^{\tau_k-\tau_{k-1}+n_0} \frac{1}{P(\Gamma)} T_*^{\tau_{k-1}}  {\pi_2}_*(P \mid \Gamma)\{\hat{R}>n\} .
		\end{aligned}
	\end{equation}
	Combining (\ref{E42}) and (\ref{E43}), we have
	\begin{equation*}
		\begin{aligned}
			P\left(\tau_{k+1}-\tau_k>n_0+n \mid \Gamma\right)
			&=T_*^{\tau_k-\tau_{k-1}+n_0}\frac{T_*^{\tau_{k-1}}(\lambda_{2} \mid \omega^{\prime})\{\hat{R}>n\}}{T_*^{\tau_{k-1}}(\lambda_{2} \mid \omega^{\prime})(T^{\tau_{k-1}}(\omega^{\prime}))} \\
			&=T_*^{\tau_k-\tau_{k-1}+n_0}\frac{T_*^{\tau_{k-1}}(\lambda_{2}\mid\omega^{\prime})(\{\hat{R}>n\}\cap T^{\tau_{k-1}}(\omega^{\prime}))}{T_*^{\tau_{k-1}}\left(\lambda_{2} \mid\omega^{\prime}\right)(T^{\tau_{k-1}}(\omega^{\prime}))}\\
		\end{aligned}
	\end{equation*}
	which together with (\ref{E28}), Lemma \ref{L26} and the fact that $ m(T^{\tau_{k-1}}(\omega^{\prime}))\geq \delta_{0} $, implies
	$$  
	\begin{aligned}
		P\left(\tau_{k+1}-\tau_k>n_0+n \mid \Gamma\right)&\leq M_0\frac{C_{1}^{2}}{\delta_{0}}m(\{\hat{R}>n\}\cap T^{\tau_{k-1}}\omega^{\prime})\\
		&\leq C_{2} m\{\hat{R}>n\},
	\end{aligned}
	$$
	with $ C_{2}=M_0{C_{1}^{2}}/{\delta_{0}}$. 
	 The case $ k $ even is similar. 
Moreover, the dependence of $C_1$ (hence $C_2)$ on $ C_{\varphi_1}', C_{\varphi_2}' $ can be removed if we take $ k $ large enough. 
	\end{proof}

\subsection{Matching $T_{*}^{n}\lambda_{1}$ with $T_{*}^{n}\lambda_{2}$. }\label{matching section}
 We are now going to estimate  $|T_{*}^{n}\lambda_{1}-T_{*}^{n}\lambda_{2}|$. 
Consider the induced dynamical system $ \widetilde{T}=(T\times T)^{S}: \Delta\times\Delta\rightarrow \Delta\times\Delta ,$ with $ S $ as in \eqref{E30}, and the functions $ 0=S_{0}<S_{1}<S_{2}<..., $ defined for each $ n\geq 1 $ as
\begin{equation*}
	S_{n}=S_{n-1}+S\circ(T\times T)^{S_{n-1}}.
\end{equation*}
Note that 
\begin{equation*}
\widetilde{T}^{n}=(T\times T)^{S_{n}}.
\end{equation*}
 Let $ \widetilde{\xi} $ be the partition of $ \Delta\times\Delta $ into the rectangles $ \Omega $ on which $S$ is constant and $ \widetilde{T} $ maps $ \Omega $ bijectively onto a union of elements of $  \eta\times\eta|_{\Delta_0 \times \Delta_0} .$ 
Without loss of generality, we assume that for any $ \Omega\in \widetilde{\xi}|_{\Delta_0 \times \Delta_0}$, there exists $ \omega_{j}\times\omega_{j^{\prime}}\in  \eta\times\eta|_{\Delta_0 \times \Delta_0}$ such that $ \Omega\subset \omega_{j}\times\omega_{j^{\prime}} .$
For each $ n\geq 1, $ define 
$$
\tilde{\xi}_n=\bigvee_{j=0}^{n-1} \widetilde{T}^{-j} (\tilde{\xi}).
$$
Each $\tilde{\xi}_n$ is a partition into sets $\Omega \subset \Delta \times \Delta$ on which $S_n$ is constant and $\widetilde{T}^n$ maps $\Omega$ bijectively onto an $\m\times\m$ mod 0 union of elements of  $ \eta\times\eta|_{\Delta_0 \times \Delta_0}$. Let us introduce a separation time in $\Delta \times \Delta$, defining for each $u, v \in \Delta \times \Delta$
$$\tilde{s}(u, v)=\min \left\{n \geq 0: \widetilde{T}^n(u)\right.\text{ and }\widetilde{T}^n(v) \text{ lie in distinct elements of }\left.\tilde{\xi}\right\}.$$
Let $ \varphi_{1} =\frac{d\lambda_{1}}{dm}, \varphi_{2}=\frac{d\lambda_{2}}{dm} $, $ \Phi(x,y)=\varphi_{1}(x)\varphi_{2}(y) $  and let $ C_{\varphi_{1}}'' $ and $ C_{\varphi_{2}}''$ be constants such that 
\begin{equation*}
	 \left|\log\frac{\varphi_{1}(x)}{\varphi_{1}(y)} \right|\leq  C_{\varphi_{1}}'' \beta^{s\left(x,y\right)} \text{ and } \left|\log\frac{\varphi_{2}(x)}{\varphi_{2}(y)} \right|\leq  C_{\varphi_{2}}'' \beta^{s\left(x,y\right)} \text{ for all } x,y \in \omega\in \eta.
\end{equation*}

\begin{lemma}\cite[Lemma 3.37]{Alves Book}\label{L32}
There exists $C_{3}>0$ depending on $C_{\varphi_{1}}'', C_{\varphi_{2}}''$  such that for all $n \geq 1, \Omega \in \tilde{\xi}_n$ and $u, v \in \widetilde{T}^{n}(\Omega) $ we have
$$
\frac{d \widetilde{T}_*^n(P \mid \Omega)}{d(m \times m)}(u) \Big/ \frac{d \widetilde{T}_*^n(P \mid \Omega)}{d(m \times m)}(v) \leq C_{3} .
$$
\end{lemma}
\begin{proposition}\label{P38}
There exists $C_{4}>0$ depending on $C_{\varphi_{1}}'', C_{\varphi_{2}}''$ such that for all $n, k \geq 0$, we have
$$
P\left\{S_{k+1}-S_k>n\right\} \leq C_{4}(m \times m)\{S>n\}.
$$
\end{proposition}
\begin{proof}
Consider first  $k=0$. Since  the density of $P$ with respect to $m \times m$ is given by $\Phi\left(x, x^{\prime}\right)=\varphi_{1}(x) \varphi_{2}\left(x^{\prime}\right)$, which is bounded from above by some constant $ C_{0}^{\prime}>0$, then we have 
$  
P\{S_{1}=S>n\}\leq C_{0}^{\prime}m\times m\{S>n\}.
$

Assume now $k \geq 1$ and take $\Omega \in \tilde{\xi}_k$. By definition $\widetilde{T}^k$ maps $\Omega$ bijectively onto a union of elements of $\eta\times\eta|_{\Delta_0 \times \Delta_0}$, and $S_{k+1}-S_k=S \circ \widetilde{T}^k$. Then, we can write

$$
\begin{aligned}
P\left(S_{k+1}-S_k>n \mid \Omega\right)=&\frac{P(\{S\circ \widetilde{T}^k>n\}\cap\Omega)}{P\left(\Omega\right)}\\
&=\frac{\tilde{T}_*^k(P \mid \Omega)\{S>n\}}{\widetilde{T}_*^k(P \mid \Omega)(\widetilde{T}^k(\Omega))}\\ 
&=\frac{\tilde{T}_*^k(P \mid \Omega)(\{S>n\}\cap\widetilde{T}^k(\Omega))}{\widetilde{T}_*^k(P \mid \Omega)(\widetilde{T}^k(\Omega))}. 
\end{aligned}
$$
Using Lemma \ref{L32}, and the fact that $ (m\times m)(\widetilde{T}^k(\Omega)) \geq \delta_{0}^{2}$, we have 
$$
P\left(S_{k+1}-S_k>n \mid \Omega\right) \leq \frac{C_{3}^2}{\delta_{0}^{2}}(m \times m)\{S>n\}.
$$
It follows that
$$
\begin{aligned}
	P\left\{S_{k+1}-S_k>n\right\} &=\sum_{\Omega \in \tilde{\xi}_k} P\left(S_{k+1}-S_k>n \mid \Omega\right) P(\Omega) \\
	& \leq \frac{C_{3}^2}{\delta_{0}^{2}} \sum_{\Omega \in \tilde{\xi}_k}(m \times m)\{S>n\} P(\Omega) \\
	&=\frac{C_{3}^2}{\delta_{0}^{2}}(m \times m)\{S>n\} P(\Delta \times \Delta),
\end{aligned}
$$
which gives us the required result with $ C_{4}={C_{3}^2}/{\delta_{0}^{2}}$. 
\end{proof}

We are going to define a sequence of densities $\tilde{\Phi}_0 \geq \tilde{\Phi}_1 \geq \tilde{\Phi}_2 \geq \cdots$ on $ \Delta\times\Delta $, which play a major role for the matching. Consider $ \epsilon>0 $ and $ i_{1}\in\mathbb N_0 $ to be choosen properly such that defining
\begin{equation}\label{E49}
\tilde{\Phi}_k(u)= \begin{cases}\Phi(u), & \text { if } k \leq i_1 \\ \tilde{\Phi}_{k-1}(u)-\epsilon J_{\widetilde{T}^k}(u) \min _{v \in \tilde{\xi}_k(u)} \frac{\tilde{\Phi}_{k-1}(v)}{J_{\tilde{T}^k}(v)}, & \text { if } k>i_1.\end{cases}
\end{equation}
the following holds (for details see~\cite[Lemma 3]{Y99}):
\begin{lemma}\label{L34}
There exists $ 0<\epsilon_{1}<1 $ such that for all $ k>i_1 $
we have 
$$\tilde{\Phi}_k\leq (1-\epsilon_{1})\tilde{\Phi}_{k-1} \text{ on } \Delta\times\Delta.  $$
\end{lemma}
\begin{remark}
It is worth mentioning that the constant $ \epsilon>0 $ in~\eqref{E49} depends on $ \beta $ and the same is true for the constant $ \epsilon_{1} $ in Lemma~\ref{L34}. Moreover, the number $i_1$ depends on $C_{\varphi_1}''$ and $C_{\varphi_2}''$.
\end{remark}

The densities $ \tilde{\Phi}_{i} $ are those of the total measure remaining in the system after $ i $ iterations by the induced map $ \widetilde{T} $. We are now going to discus the corresponding densities in real time iterations under $ T\times T $.
Let us introduce functions  $\Phi_0 \geq \Phi_1 \geq \Phi_2 \geq \cdots$ on $ \Delta\times \Delta $ as follows: for $ v\in \Delta\times \Delta $
\begin{equation*}
\Phi_n(v)=\tilde{\Phi}_i(v), \quad \text { if } \quad S_i(v) \leq n<S_{i+1}(v) .
\end{equation*}
Since, by definition, $ S_0 \leq i_1\leq S_{i_{1}}<S_{i_{1}+1} $, this implies that for all $ n\leq i_1 $, 
$ \Phi_n=\tilde{\Phi}_j$ for some  $0\leq j<i_1+1,$
and $\tilde{\Phi}_j=\Phi \text{ for all } 0\leq j<i_1+1$. Thus we have  
\begin{equation*}
\Phi_n=\Phi \text{ for all }  n\leq i_1.
\end{equation*}
Note that for all $ n\geq 1 $ we can write 
\begin{equation}\label{E64}
	\Phi=\Phi_n+\sum_{k=1}^n\left(\Phi_{k-1}-\Phi_k\right).
\end{equation}
For each $ k\geq 1 $, let $A_k=\cup_{i} A_{k, i}$,
where $A_{k, i}=\{u \in \Delta \times \Delta: k=$ $\left.S_i(u)\right\}$. Note that $A_{k, i} \cap A_{k, j}=\emptyset$ for $i \neq j$ (because $ S_i(u)\neq S_j(u) $ for $i \neq j$), and each $A_{k, i}$ is a union of elements of $\tilde{\xi}_i$.
\begin{remark}\label{R3}
By definition, for any $\Omega \in \tilde{\xi}_i|_{A_{k, i}}$, we have $ S_{i-1}|_\Omega< S_{i}|_\Omega=k $, and $S_{i-1}|_\Omega\leq k-1  $. This implies that $\Phi_{k-1}-\Phi_k=\tilde{\Phi}_{i-1}-\tilde{\Phi}_i$ on $\Omega \in \tilde{\xi}_i|_{A_{k, i}}$, and $\Phi_k=\Phi_{k-1}$ on $\Delta \times \Delta \backslash A_k$.
\end{remark}

Our main result in this subsection is
\begin{proposition}\label{first Matching Proposition}
There exists $C_{5}>0$, depending on $C_{\varphi_{1}}'', C_{\varphi_{2}}''$, such that for all $ n\geq 1$,
\begin{equation*}
\left|T_*^n \lambda_{1}-T_*^n \lambda_{2}\right|\leq	2P\{S>n\}+C_{5}\sum_{i=1}^{\infty}\left(1-\varepsilon_1\right)^i P\left\{S_i \leq n<S_{i+1}\right\}
\end{equation*}
\end{proposition}

\begin{proof}
From~\eqref{E32} and~\eqref{E64}, for each $ n\geq 1$  we have 
\begin{equation}\label{step1}
\left|T_*^n \lambda_{1}-T_*^n \lambda_{2}\right| \leq I_1+ I_2,
\end{equation}
where
 \begin{equation*}
  I_1= \left|\left({\pi_1}_*-{\pi_2}_*\right)(T \times T)_*^n\left(\Phi_n(m \times m)\right)\right|
\end{equation*}
and
\begin{equation*}
I_2=\sum_{k=1}^n\left|\left({\pi_1}_*-{\pi_2}_*\right)\left[(T \times T)_*^n\left(\left(\Phi_{k-1}-\Phi_k\right)(m \times m)\right)\right]\right|.
\end{equation*}
From~\cite[Lemma 4]{Y99} there is $K_2>0$ such that 
\begin{equation}\label{step2}
I_1\leq 2 P\{S>n\}+2K_2\left(1-\varepsilon_1\right)^{-i_1} \sum_{i=1}^{\infty}\left(1-\varepsilon_1\right)^i P\left\{S_i \leq n<S_{i+1}\right\}.
\end{equation}

Now we are going to estimate $ I_2 $. We have
\begin{equation}\label{DE72}
\begin{aligned}
I_2&=\sum_{k=1}^n\left|\left({\pi_1}_*-{\pi_2}_*\right)\left[(T \times T)_*^n\left(\left(\Phi_{k-1}-\Phi_k\right)(m \times m)\right)\right]\right|\\
&\leq 2\sum_{k=1}^n \int_{\Delta\times\Delta} \left(\Phi_{k-1}-\Phi_k\right) d(m \times m)\\
&= 2\sum_{k=1}^n \int_{\Delta\times\Delta\setminus A_k} \left(\Phi_{k-1}-\Phi_k\right) d(m \times m)+2\sum_{k=1}^n \int_{A_k} (\left(\Phi_{k-1}-\Phi_k\right) d(m \times m).
\end{aligned}
\end{equation}
From Remark \ref{R3}, we have 
\begin{equation}\label{DE73}
	2\sum_{k=1}^n \int_{\Delta\times\Delta\setminus A_k} \left(\Phi_{k-1}-\Phi_k\right) d(m \times m)=0,
\end{equation}
and writing $A_k=\cup_{i=1}^{\infty} A_{k, i}=\left(\cup_{i=1}^{i_1} A_{k, i}\right)\cup\left(\cup_{i=i_1+1}^{\infty} A_{k, i}\right)$, we have
\begin{equation}
\begin{aligned}
	\int_{A_k} \left(\tilde{\Phi}_{k-1}-\tilde{\Phi}_k\right) d(m \times m)&=\sum_{i=1}^{i_1} \int_{A_{k,i}} \left(\tilde{\Phi}_{i-1}-\tilde{\Phi}_i\right) d(m \times m)\\ &\quad +\sum_{i=i_1+1}^{\infty} \int_{A_{k,i}} \left(\tilde{\Phi}_{i-1}-\tilde{\Phi}_i\right) d(m \times m).\label{DE75}
	\end{aligned}
\end{equation}
Since by definition $ \tilde{\Phi}_i=\Phi $, for all $ 0\leq i \leq i_1 $, the first sum in the above equality vanishes.
On the other hand, by~eqref{E49} and  Lemma~\ref{L34} we have
\begin{equation}\label{DE76}
\begin{aligned}
\sum_{i=i_1+1}^{\infty} \int_{A_{k,i}} \left(\tilde{\Phi}_{i-1}-\tilde{\Phi}_i\right) d(m \times m)
&=\sum_{i=i_1+1}^{\infty} \int_{A_{k,i}} \varepsilon J_{\widetilde{T}^i}(u) \min _{v \in A_{k,i}} \frac{\tilde{\Phi}_{i-1}(v)}{J_{\tilde{T}^i}(v)} d(m \times m)\\
&\leq\varepsilon\sum_{i=i_1+1}^{\infty} \int_{A_{k,i}}   \frac{J_{\widetilde{T}^i}(u)}{J_{\tilde{T}^i}(u)}\tilde{\Phi}_{i-1}(u) d(m \times m)\\
&= \varepsilon\sum_{i=i_1+1}^{\infty} \int_{A_{k,i}}   \tilde{\Phi}_{i-1}(u) d(m \times m)\\
&\leq \varepsilon\sum_{i=i_1+1}^{\infty} \int_{A_{k,i}} (1-\varepsilon_1)^{i-i_1-1} \Phi(u) d(m \times m)\\
&=\varepsilon\sum_{i=i_1+1}^{\infty}  (1-\varepsilon_1)^{i-i_1-1}P\{S_i=k\}.\\
\end{aligned}
\end{equation}
It follows from (\ref{DE72}), (\ref{DE73}), (\ref{DE75}) and  (\ref{DE76}) that
\begin{equation}\label{DE77}
I_2\leq 2\varepsilon\sum_{k=1}^{n} \sum_{i=i_1+1}^{\infty}  (1-\varepsilon_1)^{i-i_1-1}P\{S_i=k\}=2\varepsilon \sum_{i=i_1+1}^{\infty}  (1-\varepsilon_1)^{i-i_1-1}P\left(\bigcup_{k=1}^{n}\{S_i=k\}\right).
\end{equation}
In fact, for each $ 1\leq k\leq n $, we have $ \{S_i=k\}\subset \{S_i\leq n< S_{i+1}\} $, and this implies that $ \cup_{k=1}^{n}\{S_i=k\}\subset\{S_i\leq n< S_{i+1}\} $. Then from~\eqref{DE77} we have
\begin{equation}\label{step3}
I_2\leq 2\varepsilon (1-\varepsilon_1)^{-i_1-1} \sum_{i=i_1+1}^{\infty}  (1-\varepsilon_1)^{i}P\{S_i\leq n< S_{i+1}\}.
\end{equation}
It follows from \eqref{step1}, \eqref{step2} and \eqref{step3},
\begin{equation*}
\begin{aligned}
	\left|T_*^n \lambda_{1}-T_*^n \lambda_{2}\right|&\leq 
2P\{S>n\}+C_{5}\sum_{i=1}^{\infty}\left(1-\varepsilon_1\right)^i P\left\{S_i \leq n<S_{i+1}\right\},
\end{aligned}
\end{equation*}
with $C_{5}=2(K_{2}+\frac{\varepsilon}{1-\varepsilon_1})\left(1-\varepsilon_1\right)^{-i_1}.$ We recall that $i_1$ (hence $C_5$) depends on $C_{\varphi_{1}}''$ and $C_{\varphi_{2}}''$.
\end{proof}
\begin{remark}
In~\cite{Y99} the quantity $ I_2 $ vanishes due to the full branch property. 
\end{remark}
\begin{corollary}\label{C2}
There exists $ C_{6}>0 $, depending on $C_{\varphi_{1}}'', C_{\varphi_{2}}''$ such that for all $ n\geq 1$, 
\begin{equation*}
\left|T_*^n \lambda_{1}-T_*^n \lambda_{2}\right|\leq 2 P\{S>n\}+C_{6} \sum_{i=1}^{\infty}\left(1-\varepsilon_1\right)^i (i+1)(m\times m)\left\{S>\frac{n}{i+1}\right\}.
\end{equation*}
\end{corollary}
\begin{proof}
For each $ i\geq 1 $ we have 
\begin{equation}\label{DDE79}
	P\left\{S_i \leq n<S_{i+1}\right\} \leq \sum_{j=0}^i P\left\{S_{j+1}-S_j>\frac{n}{i+1}\right\}.
\end{equation}
By Proposition \ref{P38},  
\begin{equation}\label{DE78}
	P\left\{S_{j+1}-S_j>\frac{n}{i+1}\right\}\leq C_{4}(m\times m)\left\{S>\frac{n}{i+1}\right \}.
\end{equation}
Combining (\ref{DDE79}) and (\ref{DE78}), we get
\begin{equation*}
P\left\{S_i \leq n<S_{i+1}\right\} \leq C_{4} (i+1)(m\times m)\left\{S>\frac{n}{i+1}\right\},
\end{equation*}
which, together with Proposition \ref{first Matching Proposition}, gives us the required result.
\end{proof}

\subsection{Polynomial decay}\label{Polynomial estimate} We are going to prove item (1) of  Proposition \ref{Convergence to equilibrium }.
Assume that there are $C>0$ and $a>1$ such that for all $n \geq 1$ we have $\m\{R>n\} \leq C n^{-a}$.
From  (\ref{E29}) there exists $\hat C$ such that
\begin{equation}\label{E80}
	\m\{\hat{R}>n\}=\sum_{\ell>n} \m\{R>\ell\}\leq \hat{C} n^{-a+1}.
\end{equation}
Recalling Lemma~\ref{L27} and by Lemma~\ref{L28} and (\ref{E80}) we obtain, respectively:
\begin{itemize}
\item [$(A_1)$] There exists $ \epsilon_{0}>0$, depending on $ C_{\varphi_1}', C_{\varphi_2}' $,  such that for all $ k\geq 1 $ and $ \Gamma \in\xi_{k}  $ with $ S|_\Gamma>\tau_{k-1}$, we have 
	\begin{equation*}
		P(S=\tau_{k} | \Gamma)\geq \epsilon_{0}.
	\end{equation*}
\item [$(A_2)$] For all  $n, k\geq 0$ and for all $ \Gamma \in\xi_{k}$, we have 
\begin{equation*}
	P(\tau_{k+1}- \tau_{k}>n+n_{0}| \Gamma)\leq C_{2} \hat{C}n^{-a+1}.
\end{equation*}
\end{itemize}
By~\cite[Proposition 3.46]{Alves Book} there exists $\bar C>0$ such that  
\begin{equation*}
	P\{S>n\} \leq\left(1-\epsilon_0\right)^{\left[\frac{n}{2 n_0+1}\right]}+\bar{C} \sum_{j=1}^{\infty}(j+1)^{a+1}\left(1-\epsilon_0\right)^{j-1}n^{-a+1}. 
\end{equation*}
This implies that there exists $\bar{C_{1}}>0$
\begin{equation}\label{Final estimatr of simulatainious RT}
P\{S>n\}\leq\bar{C_{1}}n^{-a+1}.
\end{equation} 
By similar arguments used to estimate $P\{S>n\}$, there exists $\bar{C_{2}}>0$ such that 
\begin{equation}\label{estimatr of simulatainious RT with mxm}
	(m\times m)\{S>\frac{n}{i+1}\}\leq\bar{C_{2}} \left(\frac{n}{i+1}\right)^{-a+1}.
\end{equation}
 By using the estimates \eqref{Final estimatr of simulatainious RT} and \eqref{estimatr of simulatainious RT with mxm} in  Corollary \ref{C2}, we get
\begin{equation}\label{estimate of Tlambda_1-Tlambda_2 }
	\left|T_*^n \lambda_{1}-T_*^n \lambda_{2}\right|\leq C^{\prime}  n^{-a+1} \text{ for some } C^{\prime}>0.
\end{equation}
 Considering $ \lambda_{1}=\lambda $ and $\lambda_{2}=\nu $ in \eqref{estimate of Tlambda_1-Tlambda_2 } we conclude
\begin{equation*}
	\left|T_*^n \lambda-\nu\right|\leq C^{\prime}n^{-a+1}.
\end{equation*}

\subsection{(Stretched) Exponential decay}\label{Exponential estimates} We are now going to prove the remaining item (2) of  Proposition~\ref{Convergence to equilibrium }.
Assume that there are $C, c>0$ and $0<a \leq 1$ such that $m\{R>n\} \leq C e^{-c n^a}$ for all $n \geq 1$. Then, there exists $\hat{C}>0$ such that
\begin{equation*}
	m\{\hat{R}>n\}=\sum_{\ell>n} m\{R>\ell\} \leq \hat{C} e^{-c n^a},
\end{equation*}
which together with Lemma~\ref{L27}, Lemma~\ref{L28} and \cite[Proposition 3.48]{Alves Book} gives
\begin{equation*}
	P\{S>n\}\leq \bar{C_{3}} e^{-c_{0} n^a} \text{ for some } \bar{C_{3}}>0.
\end{equation*} 
By similar arguments as used in polynomial case we conclude that
\begin{equation*}
	\left|T_*^n \lambda-\nu\right|\leq C^{\prime} e^{-c^{\prime} n^a} \text{ for some } C^{\prime}, c^{\prime}>0.
\end{equation*}
\begin{remark}
	Recall that the constant $ \epsilon_{1} $ depends on $ \beta $. This does not affect exponent $a$ in the polynomial case. On the other hand, this causes the exponent $ c^{\prime} $ to depend on $ \beta $ in the (stretched) exponential case. 
\end{remark}

\section{Central Limit Theorem}\label{sec:CLT}

In this section we prove Corollary~\ref{CLT Main Cror}. We start by recalling a version of~\cite[Theorem~1.1]{Liverane CLT theorem}. 
\begin{theorem}\label{Liverani CLT}
Let $(X, \mathcal{F}, \upsilon)$ be a probability space, $T\colon X \to X$ a (non-invertible) measurable map and $ \upsilon $ an ergodic $ T$-invariant probability measure. Let $\varphi \in {L}^{\infty}(X, \upsilon)$ be such that $\int \varphi d \upsilon=$ 0. Assume
\begin{itemize}
\item[i)] $ \sum_{n=1}^{\infty}\left|\int\left(\varphi \circ T^n\right) \varphi d \upsilon\right|<\infty $
\item[ii)] $ \sum_{n=1}^{\infty}\left(\hat{T}^{* n} \varphi\right)(x) \text { is absolutely convergent for }\upsilon \text {-a.e. } x, $
\end{itemize}
where $\hat{T}^*$ is the dual of the operator $\hat T\varphi = \varphi \circ T$. Then the Central Limit Theorem holds for $\varphi$ if and only if $\varphi$ is not coboundary.
\end{theorem}
The dual operator $\hat{T}^*$ above is the Perron-Frobenius operator with respect to $T$ and $\upsilon$, that is
$$
(\hat{T}^* \varphi)(x)=\sum_{y\in T^{-1}(x)} \frac{\varphi(y)}{J_T(y)},
$$
where $J_{T}$ here is defined in terms of the measure $\upsilon$.

\begin{proposition}\label{CLT P1}
	Let $ T:\Delta\rightarrow\Delta $ be the tower map of an aperiodic induced WGM expanding map $f^{R}$ with a coprime block, $R \in L^{1}(\m)$, and let $\nu $ be the unique ergodic $ T$-invariant probability measure. If $m\{ R >n \} \leq C n^{-a}$ for some $C>0$ and $a >2$, then the CLT is satisfied for all $\varphi \in \mathcal{F}_{\beta}(\Delta)$ if and only if $\varphi$ is not coboudary.
\end{proposition}
\begin{proof}
From $ii)$ in Lemma~\ref{D.Correlation concection of T and convergence of equlibrium} and Proposition~\ref{Convergence to equilibrium } there is some $ C^{\prime}>0 $ such that 
\begin{equation}\label{summable decay}
\Cor_{\nu}(\varphi, \varphi \circ T^{n})\leq C^{\prime} n^{-a+1}.
\end{equation}
 Let $\phi=\varphi-\int \varphi d \nu$, so that $\int \phi d \nu=0$. We are going to show that $\phi$ satisfies  conditions i) and ii)  of the Theorem \ref{Liverani CLT}. It is straightforward to check that $$ \Cor_{\nu}(\varphi, \varphi \circ T^{n})=\Cor_{\nu}(\phi, \phi \circ T^{n})=\left|\int(\phi \circ T^n) \phi d \nu\right|.$$ It follows from \eqref{summable decay} that condition i) holds.
Since $m$ and $\nu$ are equivalent measures, it suffices to verify condition ii) for $m$-a.e. $x\in\Delta$. The operator $\hat{T}^*$ is defined in terms of the invariant measure, so for a measure $\lambda \ll m$ it sends $\frac{d \lambda}{d \nu}$ to $\frac{d T_* \lambda}{d \nu}$. We can write
$$
(\hat{T}^{* n} \phi)(x)=\frac{1}{\rho(x)}(\hat{P}^n(\phi \rho))(x),
$$
where $\hat{P}$ is the Perron-Frobenius operator with respect to $m$, that sends  $\frac{d \lambda}{d m}$ to $\frac{d T_* \lambda}{d m}$.
We shall now write $\phi$ as the difference of the densities of two (positive) measures of similar regularity to $\phi$. We let $\tilde{\phi}=b(\phi+a)$, for some large $a$, with $b>0$ chosen such that $\int \tilde{\phi} \rho d m=1$. We define the probability measures $\lambda_{1}, \lambda_{2}$ by
$$
\frac{d \lambda_{1}}{d m}=(b \phi+\tilde{\phi}) \rho,\quad \frac{d \lambda_{2}}{d m}=\tilde{\phi} \rho .
$$
Similar to Lemma~\ref{lemma:phi*} we can chech that $\frac{d \lambda_{1}}{d m}, \frac{d \lambda_{2}}{d m} \in\mathcal{F}_{\beta}^{+}(\Delta) $. 
Moreover,
$$
b^{-1}\left(\frac{d \lambda_{1}}{d m}-\frac{d \lambda_{2}}{d m}\right)=\phi \rho.
$$
Following Section~\ref{matching section}, for these given measures $\lambda_{1}$ and $\lambda_{2}$  we have
\begin{equation*}
	\begin{aligned}
	T_*^n \lambda_{1}-T_*^n \lambda_{2}&= {\pi_1}_*(T \times T)_*^n\left([\Phi_n+\sum_{k=1}^n(\Phi_{k-1}-\Phi_k)](m \times m)\right)\\
	&+{\pi_2}_*(T \times T)_*^n\left([\Phi_n+\sum_{k=1}^n(\Phi_{k-1}-\Phi_k)](m \times m)\right).
	\end{aligned}
\end{equation*}
Let $\psi_n, \psi_n'$ be the densities with respect to $m$ of the first and the second term above, respectively. Since $\hat{P}$ is a linear operator, we see that
$$
\begin{aligned}
	|\hat{P}^n(\phi \rho)|&=\left|\hat{P}^n\left(b^{-1}\left(\frac{d \lambda_{1}}{d m}-\frac{d \lambda_{2}}{d m}\right)\right)\right|\\
	& =b^{-1}\left|\frac{d T_*^n \lambda_1}{d m}-\frac{d T_*^n \lambda_2}{d m}\right| \\
	& \leq b^{-1}\left(\psi_n+\psi_n^{\prime}\right).
\end{aligned}
$$
Then by using above estimates on $ I_{1}, I_{2} $ and estimates in Section~\ref{Polynomial estimate}, we have $$\int \psi_n d m=\int \psi_n^{\prime} d m=\int (\Phi_n+\sum_{k=1}^n(\Phi_{k-1}-\Phi_k)) d(m \times m)\leq Cn^{-a+1},$$  for some $ C>0 $. This implies that $$\sum_{n=1}^{\infty}\int_{ \Delta}|\hat{P}^n(\phi \rho)|dm \leq C^{\prime}\sum_{n=1}^{\infty}n^{-a+1} < \infty,$$ for some $ C^{\prime}>0.$  We must have that $\sum_{n=1}^{\infty}|\hat{P}^n(\phi \rho)(x)|$ is convergent for  $ m$-a.e. $x$. This implies that $$\sum_{n=1}^{\infty}|(\hat{T}^{* n} \phi)(x)|=\frac{1}{\rho(x)}\sum_{n=1}^{\infty}|\hat{P}^n(\phi \rho)(x)|$$ is convergent for  $ m$-a.e. $x$, and condition ii) holds.
\end{proof}

Since $\varphi\in\mathcal H_\eta$, by Lemma~\ref{lemma: phi circ pi}, $\varphi\circ \pi\in\mathcal F_{\beta^\eta}(\Delta)$. Notice that $\varphi$ is not coboundary if and only if $\varphi\circ \pi$ is not coboundary. For every interval $ J\subset \mathbb{R} $,
\begin{equation*}
\begin{aligned}
	\pi_{*}\nu &\left\{ y\in M : \frac{1}{\sqrt{n}}\sum_{i=0}^{n-1}\left(\varphi( f^{i}(y))- \int \varphi d\pi_{*}\nu \right) \in J \right\}\\
	&=\nu \left\{ x\in \Delta : \frac{1}{\sqrt{n}}\sum_{i=0}^{n-1}\left(\varphi( f^{i}(\pi(x)))- \int \varphi d\pi_{*}\nu \right)\in J \right\}\\
	&=\nu \left\{ x\in \Delta : \frac{1}{\sqrt{n}}\sum_{i=0}^{n-1}\left(\varphi\circ\pi(T^i(x))- \int \varphi\circ\pi d\nu \right)\in J \right\}.
\end{aligned}
\end{equation*}	
The proof of Corollary~\ref{CLT Main Cror} follows now from Proposition~\ref{CLT P1} since $\pi_*\nu=\mu$.

\bibliographystyle{alpha}

\end{document}